%% file: main.tex
\documentclass[12pt]{amsart}

\input{paper-style}

\title{Moments of random multiplicative functions over function fields}

\author{Maximilian C. E. Hofmann}
\address{Maximilian C. E. Hofmann \\
Institut f\"ur Mathematik \\
Goethe--Universit\"at Frankfurt \\
60325 Frankfurt am Main \\
Germany}
\email{maximilian.hofmann@stonybrook.edu}

\author{Annemily Hoganson} 
\address{
Annemily Hoganson \\
Department of Mathematics and Statistics \\
Carleton College \\ 
Northfield, MN, 55057 \\
USA}
\email{annemilyhoganson@gmail.com}

\author{Siddarth Menon}
\address{
Siddarth Menon \\
Department of Mathematics \\
University of California Berkeley \\ 
Berkeley, CA, 94704 \\
USA}
\email{sidmenon@berkeley.edu}

\author{William Verreault}
\address{
William Verreault \\
Department of Mathematics \\
University of Toronto   \\ 
Toronto, ON, M5S 2E4 \\
Canada}
\email{william.verreault@utoronto.ca}

\author{Asif Zaman}
\address{
Asif Zaman \\
Department of Mathematics \\ 
University of Toronto   \\ 
Toronto, ON, M5S 2E4 \\
Canada}
\email{asif.zaman@utoronto.ca}

\date{\today}

\begin{document}

\begin{abstract}
Granville--Soundararajan, Harper--Nikeghbali--Radziwi\l\l, and Heap--Lindqvist independently established an asymptotic for the even natural moments of partial sums of random multiplicative functions defined over integers. Building on these works, we study the even natural moments of partial sums of Steinhaus random multiplicative functions defined over function fields. Using a combination of analytic arguments and combinatorial arguments, we obtain asymptotic expressions for all the  even natural moments in the large field limit and large degree limit, as well as an exact expression for the fourth moment. 
\end{abstract}

\maketitle


\section{Introduction}{\label{intro}}
\subsection{Background} Let $(f(p))_{p \text{ prime}}$ be a sequence of independent random variables which are uniformly distributed on the complex unit circle. Define  $f \colon \mathbb{N} \to \mathbb{C}$ by  $f(1) = 1$ and
\[
f(n) = f(p_1)^{\alpha_1} \cdots f(p_r)^{\alpha_r}, \qquad 
\]
where $n = p_1^{\alpha_1} \cdots p_r^{\alpha_r}$ for distinct primes $p_1,\dots,p_r$ and positive integers $\alpha_1,\dots,\alpha_r$. The function $f$ is a \textit{Steinhaus random multiplicative function} (over the integers $\mathbb{Z}$).  Random multiplicative functions were first introduced by Wintner \cite{Wintner-1944} in 1944 as a probabilistic model for the M\"{o}bius function in multiplicative number theory.  The  Steinhaus example is inspired by archimedean characters $n \mapsto n^{it}$ for randomly chosen large values of $t$, as well as by non-archimedean characters $n \mapsto \chi(n)$ for randomly chosen Dirichlet characters $\chi$ modulo $q$ for large prime moduli $q$ (see Section 2 of \cite{GranvilleSound} for a related discussion). 

For fixed $x \geq 2$, the random variable 
\begin{equation}
\sum_{n\leq x} f(n)
\label{eqn:SteinhausChaos-Integer}
\end{equation}
and related subsums have generated substantial interest over the past decade, e.g.\ \cite{Hough-2011}, \cite{ChatterjeeSound-2012}, \cite{Harper-2013}, \cite{Lau-Tenenbaum-Wu-2012}, \cite{HarperNikeghbaliRadziwi-2015}, \cite{BondarenkoSeip-2016}, \cite{HeapLindqvist-2016},  \cite{Harper-2019a}, \cite{Harper-2020}, \cite{Mastrostefano-2022}, \cite{Caich-2023}, \cite{Hardy-2024a}, \cite{Harper-2023},  \cite{Klurman-Shkredov-Xu-2023}, \cite{SoundXu-2023}, \cite{GorodetskyWong-2024b}, \cite{GorodetskyWong-2024a}, \cite{Hardy-2024b}, \cite{Pandey-Wang-Xu-2024}, and \cite{Xu-2024}. While  the distribution of \eqref{eqn:SteinhausChaos-Integer} remains open (see Gorodetsky--Wong \cite{GorodetskyWong-2024b} for a recent conjecture), there has been significant progress to describe its moments. Granville--Soundararajan \cite{GranvilleSound} first gave upper and lower bounds for even natural moments of  \eqref{eqn:SteinhausChaos-Integer}. About 15 years later, Granville--Soundararajan (unpublished), Harper--Nikeghbali--Radziwi\l\l\, \cite{HarperNikeghbaliRadziwi-2015}, and Heap--Lindqvist \cite{HeapLindqvist-2016} independently proved that for $k \in \N$ and $x \geq 3$,
\begin{equation}\label{eqn:Moments-Integers}
\mathbb{E} \Big| \sum_{n\leq x} f(n)\Big|^{2k} = \binom{2k-2}{k-1} c_k \cdot k^{-(k-1)}\mathrm{vol}(\mathcal{B}_k) x^k (\log x)^{(k-1)^2} \Big(1 + O_k\Big(\frac{1}{\log\log x}\Big)\Big),
\end{equation}
where $\text{vol}(\mathcal{B}_k)$ is the volume of the $k$th Birkhoff polytope, and 
\begin{equation} \label{eqn:ArithmeticConstant-Integers}
c_k = \prod_p \Big(1-\frac{1}{p} \Big)^{k^2} \Big( \sum_{m=0}^{\infty} {m+k-1 \choose k-1}^2 \frac{1}{p^m} \Big). 
\end{equation}
About 10 years prior, Conrey--Gamburd \cite{ConreyGamburd-2006} established essentially the same estimate as \eqref{eqn:Moments-Integers} in the equivalent context of integral moments of partial sums of the Riemann $\zeta$-function along the critical line. Morever, the special case $k=2$ was studied earlier by Ayyad--Cochrane--Zheng \cite{AyyadCochraneZheng-1996} and improved by Shi \cite{Shi-2008} who showed that
\begin{equation}
    \label{eqn:integers_k=2}
    \mathbb{E}\Big| \sum_{n \leq x} f(n) \Big|^4 = \frac{12}{\pi^2} x^2 \log x + c x^2 + O_{\epsilon}(x^{547/416+\epsilon})
\end{equation}
for $\epsilon > 0$, where $c = 0.511\dots$ is an explicit positive constant.   Recently, Harper \cite{Harper-2019a,Harper-2020} determined the order of magnitude of $\mathbb{E}|\sum_{n \leq x} f(n)|^{2k}$  for all real numbers $k \geq 0$.

Our interest lies with random multiplicative functions in the \textit{function field setting}, namely the polynomial ring $\mathbb{F}_q[t]$ where $q \geq 2$ is a prime power and $\mathbb{F}_q$ is the finite field with $q$ elements. For an introduction to multiplicative functions over $\mathbb{F}_q[t]$, see Granville--Harper--Soundararajan \cite{GranvilleHarperSoundararajan-2015} for example. Before presenting our results, we define some notation. Let $\mathcal{M}$ denote the set of all monic polynomials in $\mathbb{F}_q[t]$. Let $\mathcal{M}_N \subseteq \mathcal{M}$ denote the set of monic polynomials of degree exactly $N$, so $|\mathcal{M}_N|=q^N$. Let $(f(P))_{P}$ be a sequence of independent random variables indexed by the monic irreducible polynomials $P$ of $\mathbb{F}_q[t]$ and uniformly distributed on the complex unit circle. Define $f \colon \mathcal{M} \to \mathbb{C}$ by $f(1) = 1$ and 
\[
f(F) = f(P_1)^{\alpha_1} \cdots f(P_r)^{\alpha_r}
\]
for $F \in \mathcal{M}$, where $F = P_1^{\alpha_1} \cdots P_r^{\alpha_r}$ for distinct monic irreducible polynomials $P_1,\dots,P_r$ and positive integers $\alpha_1,\dots,\alpha_r$. The function $f$ is a \textit{Steinhaus random multiplicative function} over the ring $\mathbb{F}_q[t]$. For integers $N \geq 1$, the random variable
\begin{equation}
\sum_{F \in \mathcal{M}_N}   f(F)
\label{eqn:SteinhausChaos}
\end{equation}
is the function field analogue of \eqref{eqn:SteinhausChaos-Integer}.   Building on work of Harper \cite{Harper-2013} over the integers, Aggarwal--Subedi--Verreault--Zaman--Zheng \cite{AggarwalSubediVerreaultZamanZheng-2022a} established a central limit theorem for a variant of \eqref{eqn:SteinhausChaos} with few irreducible factors. Otherwise, much less is known about \eqref{eqn:SteinhausChaos}. 

\subsection{Results} The purpose of this article is to obtain asymptotics for the even natural moments of \eqref{eqn:SteinhausChaos}. Our main result establishes such an estimate as $q^N \to \infty$. 

\begin{theorem}\label{thm:steinhaus}
    Fix $k\in \N$. For any prime power $q \geq 2$ and any integer $N \geq 1$, if $f$ is a Steinhaus random multiplicative function defined over $\mathbb{F}_q[t]$, then  
    \begin{equation} \label{eqn:Moments}
        \mathbb{E}\Big|\sum_{F\in \cM_N} f(F)\Big|^{2k} = b_k(q) \mathcal{S}_k(N) q^{kN} 
    \left(
    1 + O_k\left(\frac{(\log N)^{k}}{q N^{\frac{1}{2} - \frac{3k-2}{4k^2-2k}} }  
    \right) \right),
    \end{equation}
    where $\mathcal{S}_k(N)$ is the number of $k \times k$ magic squares with magic constant $N$, and 
    \begin{equation} \label{eqn:ArithmeticConstant}
        b_k(q) = \prod_{\substack{P \in \cM\\ \text{irred.}}} \Big( 1-\frac{1}{q^{\deg P}}\Big)^{k^2} \Big( \sum_{m=0}^\infty \binom{m+k-1}{k-1}^2 \frac{1}{q^{m \deg P}}  \Big).
    \end{equation}
\end{theorem}

\begin{remark}
    A $k \times k$ magic square with magic constant $N$ is a $k \times k$ matrix with nonnegative integer entries where every row sum and column sum is equal to $N$. For fixed $k$, it is a consequence of Ehrhart's theorem \cite{Ehrhart-1977} that as $N \to \infty$,
    \begin{equation}
    \mathcal{S}_k(N) = k^{-(k-1)}\text{vol}(\mathcal{B}_k)N^{(k-1)^2} (1+o(1)). 
    \label{eqn:MagicSquaresAsymptotic}
    \end{equation}
\end{remark}
The proof of \cref{thm:steinhaus} in \cref{sec:analytic-general} mostly follows the analytic approach of \cite{HarperNikeghbaliRadziwi-2015} and especially \cite{HeapLindqvist-2016} with some novelties, so we shall provide a brief comparison with \eqref{eqn:Moments-Integers}. First, the asymptotic of \cref{thm:steinhaus} over $\mathbb{F}_q[t]$ and the asymptotic \eqref{eqn:Moments-Integers} over $\mathbb{Z}$ possess similar main terms. Recall the number of integers $n\leq x$ is equal to  $\lfloor x\rfloor$, and the number of polynomials $F \in \mathcal{M}_N$ is equal to $q^N$. By replacing $x$ in \eqref{eqn:Moments-Integers} by $q^N$ and utilizing \eqref{eqn:MagicSquaresAsymptotic}, one recovers the same growth rate. Moreover, the arithmetic constants \eqref{eqn:ArithmeticConstant-Integers} and \eqref{eqn:ArithmeticConstant} have the same structure.

Second, the error term of \cref{thm:steinhaus} over $\mathbb{F}_q[t]$ is comparably stronger than \eqref{eqn:Moments-Integers} over $\mathbb{Z}$. Replacing $x$ with $q^N$ again, one might likewise expect the $\log \log x$ factor in \eqref{eqn:Moments-Integers} to become a factor of $\log N$  in the function field setting, but our error term decays by a power of $N$. These gains are primarily due to a more delicate technical analysis of various oscillatory integrals, and might lead to similar gains in the integer setting.  

Third, \cref{thm:steinhaus} has complete uniformity in both the $q$-limit and $N$-limit. Our adaptation of the analytic approach due to \cite{HarperNikeghbaliRadziwi-2015,HeapLindqvist-2016} handles the case $q^N \to \infty$ provided $N \geq 20$. With more effort, a modified analytic approach might allow for all values of $N$. Instead, for $N \geq 1$ fixed and $q \to \infty$, we applied a combinatorial approach to prove this special case of \cref{thm:steinhaus}. We state this theorem separately for the sake of clarity. 

\begin{restatable}{theorem}{qlimit thm}\label{thm:q-limit}
    Fix $k, N\in \N$. For any prime power $q \geq 2$, if $f$ is a Steinhaus random multiplicative function defined over $\mathbb{F}_q[t]$, then   $$\mathbb{E}\Big|\sum_{F\in \cM_N} f(F)\Big|^{2k} = \mathcal{S}_k(N) q^{kN}\Big(1 + O_{k,N}\Big(\frac{1}{q}\Big) \Big).$$
\end{restatable}
\begin{remark}
   The main term matches \eqref{eqn:Moments} since $b_k(q) = 1 + O_k(q^{-1})$.
\end{remark}

The proof of \cref{thm:q-limit} in \cref{sec:combinatorial-general} adapts the GCD matrix construction first introduced by Vaughan--Wooley \cite{VaughanWooley-1995} and used by Granville--Soundararajan \cite{GranvilleSound} for moments in the integer setting. Gorodetsky \cite[Section 5.1]{Gorodetsky-2024} outlined this strategy to establish an asymptotic for moments in the function field setting as $q \to \infty$. Our proof follows his outline. 

By specializing the analytic (\cref{thm:steinhaus}) and combinatorial (\cref{thm:q-limit}) approaches to the case $k=2$, we provide two proofs of an exact formula for the fourth moment. 
\begin{restatable}{theorem}{kistwo thm}\label{thm:k=2}
    For any prime power $q \geq 2$ and integer $N \geq 1$, if $f$ is a Steinhaus random multiplicative function defined over $\mathbb{F}_q[t]$, then 
    $$\mathbb{E}\Big|\sum_{F\in \cM_N}f(F)\Big|^4 = Nq^{2N}\Big(1-\frac{1}{q}\Big)+q^{2N}.$$
\end{restatable}
These two short proofs appear in \cref{sec:combinatorial-4,sec:analytic-4}. In addition to establishing \cref{thm:k=2}, they also have the advantage of outlining the key steps leading to \cref{thm:steinhaus,thm:q-limit} in a simple case. This exact formula in the function field setting can be compared with the asymptotic \eqref{eqn:integers_k=2} in the integer setting. 

\subsection{Organization} 
\cref{sec:Preliminaries} fixes notation, prepares the starting point for all of our main proofs, and includes some discussion of  our methods compared to the proof of \eqref{eqn:Moments-Integers}. \cref{sec:combinatorial-4,sec:analytic-4} respectively contain the combinatorial proof and the analytic proof of the 4th moment formula in \cref{thm:k=2}. These expositions also serve as simplified introductions for the general case with each approach. \cref{sec:combinatorial-general} generalizes the combinatorial approach to all $2k$th moments in the $q$-limit and proves \cref{thm:q-limit}. 

The rest of the paper (\cref{sec:analytic-general,sec:analytic-B-convergence,sec:analytic-integral-repn,sec:analytic-magic-square,sec:analytic-error-integrals}) is dedicated to the proof of \cref{thm:steinhaus} for all $2k$th moments via an analytic approach. \cref{sec:analytic-general} proves \cref{thm:steinhaus} assuming several key propositions and proceeding with the strategy of \cref{sec:analytic-4} for $k=2$.  The key propositions are postponed to the subsequent sections. \cref{sec:analytic-B-convergence} establishes analytic theory for the relevant multivariable generating function and its convergence (\cref{prop:convergence_of_B}). \cref{sec:analytic-integral-repn} rewrites the $2k$th moment as an integral in terms of this generating function and performs a convenient change of variables (\cref{prop:integral_representation_of_2kth_moment}). \cref{sec:analytic-magic-square} extracts the main contribution from this integral and expresses its value in terms of the number of magic squares (\cref{prop:magic_square_integral}). \cref{sec:analytic-error-integrals} handles the remaining error term integral by carefully subdividing the region of integration to control the fluctuations in the integrand (\cref{prop:error_term_estimation}). The bulk of the contribution comes from a high peak on a narrow region, so this analysis is fairly delicate.

\section{Preliminaries}
\label{sec:Preliminaries}

\subsection{Notation} For the sake of clarity, we fix some standard notation and conventions. This repeats some discussion from the introduction. Let $q \geq 2$ be a prime power, and let $\mathbb{F}_q[t]$ be the polynomial ring over the finite field $\mathbb{F}_q$ with $q$ elements. Let $\mathcal{M}$ denote the set of all monic polynomials in $\mathbb{F}_q[t]$ and let $\mathcal{M}_N \subseteq \mathcal{M}$ be the subset of degree $N$ polynomials, so $|\mathcal{M}_N| = q^N$.   The letters $F$ and $G$  will always denote monic  polynomials, and the letters $P$ and $Q$ will always denote monic irreducible polynomials.  The zeta function $\zeta_q$ for function fields is defined as
\begin{align*}
    \zeta_q(z) &= \sum_{N=0}^{\infty} |\mathcal{M}_N| z^N = \prod_{P}\left(1-z^{\deg P}\right)^{-1} = \frac{1}{1-qz},
\end{align*} 
where products over $P$ will always run over all irreducible monic polynomials.  

 A magic square of size $k$ with magic constant $N$ is a $k\times k$ matrix whose entries are nonnegative integers such that the entries of each row and column sum to $N$. The number of such matrices will be denoted by $\mathcal{S}_k(N)$. 

For variables $a$ and $b$, we write $a \ll b$ or $a = O(b)$ to say that there exists an absolute positive constant $C$ such that $|a| \leq C b$. If the constant $C$ depends on a parameter, say $k$, we shall write $a \ll_k b$ or $a = O_k(b)$. If $a$ and $b$ depend on a positive parameter $x$, then we say that $a = o(b)$ as $x \to \infty$  if the ratio $\frac{a}{b}$ converges to $0$ as $x \to \infty$.

\subsection{Setup} For the entirety of the paper, let $f \colon \mathcal{M} \to \mathbb{C}$ be a Steinhaus random multiplicative function over $\mathbb{F}_q[t]$. Our goal is to calculate the $2k$th moments of \eqref{eqn:SteinhausChaos} for integers $k$. All of our proofs begin by reducing to a natural counting problem. By complete multiplicativity and independence of $f$ on distinct irreducibles $P$, it follows for any $F,G \in \mathcal{M}$ that 
\[
\mathbb{E}[f(F)\overline{f(G)}] = \mathbbm{1}_{\{F = G\}},
\]
where $\mathbbm{1}_{\mathcal{A}}$ denotes the indicator function of a set $\mathcal{A}$. For the case $k=1$, this implies that 
$$
\mathbb{E}\Big| \sum_{F\in \cM_N} f(F)\Big|^{2} = \dsum_{F, G\in \cM_N}\mathbb{E}[f(F)\overline{f(G)}] = \dsum_{F, G\in \cM_N}\mathbbm{1}_{\{F=G\}}  = |\cM_N|=q^N.
$$
Similarly, the $2k$th moments are given by
$$
\msum_{F_1, \dots, F_{2k}\in \cM_N} \mathbb{E}[f(F_1)\dots f(F_k) \overline{f(F_{k+1})\dots f(F_{2k})}] = \msum_{F_1, \dots, F_{2k}\in \cM_N} \mathbbm{1}_{\{F_1\dots F_k = F_{k+1}\dots F_{2k}\}}.
$$
leading to the key initial identity: 
\begin{equation} \label{eqn:setup}
\mathbb{E}\Big| \sum_{F \in \mathcal{M}_N} f(F) \Big|^{2k} = \#\Big\{(F_1, \dots, F_{2k})\in \mathcal{M}_N^{2k}\ \Big|\ F_1F_2\cdots F_k = F_{k+1}F_{k+2}\cdots F_{2k}\Big\}. 
\end{equation}


\subsection{Remarks on the integer setting}
The moment problem over the integers can be reduced to a counting problem similar to \eqref{eqn:setup}. More precisely, the left hand side of \eqref{eqn:Moments-Integers} is equal to
\begin{equation} \label{eqn:setup-integers}
\# \Big\{ (n_1,\dots,n_{2k}) \in \{1,2,\dots,\lfloor x \rfloor \}^{2k} \ \Big|\  n_1\cdots n_k = n_{k+1} \cdots n_{2k} \Big\}. 
\end{equation}
Estimating this quantity takes several approaches in the literature.  Granville--Soundararajan \cite{GranvilleSound} give upper and lower bounds for \eqref{eqn:setup-integers} via their GCD matrix construction. Our proof of \cref{thm:q-limit} applies this method over $\mathbb{F}_q[t]$ to produce an asymptotic in the $q$-limit; see also \cref{sec:combinatorial-4} for the case $k=2$. 

Harper--Nikeghbali--Radziwi\l\l \, \cite{HarperNikeghbaliRadziwi-2015} and Heap--Lindqvist \cite{HeapLindqvist-2016} count \eqref{eqn:setup-integers} using the analytic theory of multivariable Dirichlet series. They enforce the condition $1 \leq n_i \leq N$ with a suitably smooth ramp function, which is written as $2k$ line integrals via Mellin inversion. This essentially rewrites \eqref{eqn:setup-integers} as an iterated integral over the multivariable Dirichlet series
\begin{equation} \label{eqn:GeneratingFunction-integers}
\msum_{\substack{n_1, n_2, \dots, n_{2k} \\ n_1 \cdots n_k = n_{k+1} \cdots n_{2k}}} \frac{1}{n_1^{s_1} \cdots n_{2k}^{s_{2k}}}.
\end{equation}
By expressing this Dirichlet series as a product of Riemann zeta functions and an Euler product with an expanded region of absolute convergence, they can extract the integral's asymptotic size. Roughly speaking, after a change of variables of the form $s_j \mapsto \frac{1}{2}+it_j$, they show that the main contribution is concentrated when $t_j \in \R$ is very small for many $j$.  This truncation requires a careful case analysis by subdviding the remaining large region of integration, where the complicated integrand decays rapidly as $t_j$ grows for some fixed $j$. 

Our proof of \cref{thm:steinhaus} in \cref{sec:analytic-general}  adapts this strategy to the function field setting.  The multivariable Dirichlet series \eqref{eqn:GeneratingFunction-integers} over the integers corresponds to a multivariable power series  \eqref{eqn:GeneratingFunction} over $\mathbb{F}_q[t]$. As established in \cref{sec:analytic-B-convergence}, this power series will also factor as a product of zeta functions and another series with expanded radius of convergence. By construction, the quantity \eqref{eqn:setup} is a specific coefficient of our multivariable power series, expressed as an iterated integral via Cauchy's integral formula. The  case $k=2$ for \cref{thm:k=2} yields a sufficiently simple series factorization and integral expression, which we compute directly in \cref{sec:analytic-4}. For $k \geq 3$, we parameterize the resulting integrals, and perform a truncation like  Heap--Lindqvist. After an analogous change of variables in \cref{sec:analytic-integral-repn}, the main contribution near the origin is extracted  in \cref{sec:analytic-magic-square}.  The proof concludes in \cref{sec:analytic-error-integrals}  with a technical analysis of the integral over the complement of the truncated region.


\section{Fourth moment computation using a combinatorial approach}
{\label{sec:combinatorial-4}}

We provide a short proof of \cref{thm:k=2} by counting coprime pairs of monic polynomials of fixed degree. We require the following lemma.

\begin{lemma}[Proposition 3 in \cite{CORTEEL1998186}]{\label{lem:corteel}}
    The number of coprime pairs of monic polynomials of degree exactly $N\geq 1$ over $\mathbb{F}_q[t]$ is $q^{2N}-q^{2N-1}$. 
\end{lemma}

\begin{proof}[Proof of \cref{thm:k=2}]
By \eqref{eqn:setup}, the fourth moment is given by
$$\qsum_{\substack{F_1, F_2, G_1, G_2 \in \mathcal{M}_N \\ F_1 F_2=G_1 G_2}}1.$$
We note that $F_1F_2 = G_1 G_2$ if and only if
\begin{align*}
    \frac{F_1}{(F_1, G_2)}\frac{F_2}{(F_2, G_1)} = \frac{G_1}{(F_2, G_1)}\frac{G_2}{(F_1, G_2)},
\end{align*}
where $(F,G)$ denotes the \textit{GCD} of $F$ and $G$,
which is equivalent to $$\frac{F_1}{(F_1, G_2)}=\frac{G_1}{(F_2, G_1)} \text{ and }\frac{F_2}{(F_2, G_1)} = \frac{G_2}{(F_1, G_2)}.$$
 Letting $A := \frac{F_1}{(F_1, G_2)}$ and $B := \frac{G_2}{(F_1, G_2)}$, we observe that $A$ and $B$ can be any monic, coprime polynomials such that $$\deg A = \deg B = N-\deg(F_1, G_2) = N-\deg(F_2, G_1),$$ while $(F_1, G_2)$ and $(F_2, G_1)$ can be any monic polynomials of degree $N-\deg A = N-\deg B$. Thus, counting the possible values of $(F_1, G_2)$ and $(F_2, G_1)$ for each possible value of $A$ and $B$, we have \begin{align*}
    \qsum_{\substack{F_1, F_2, G_1, G_2 \in \mathcal{M}_N \\ F_1 F_2=G_1 G_2}}1&= \dsum_{\substack{(A,B)=1\\ \deg A = \deg B \leq N}}\Big(\sum_{F \in \cM_{N-\deg A}} 1\Big)^2\\
    &= \dsum_{\substack{(A,B)=1\\ \deg A = \deg B \leq N}}\Big(q^{N-\deg A}\Big)^2\\
    &= \sum_{m=0}^N q^{2(N-m)} \dsum_{\substack{(A,B)=1\\ \deg A=\deg B=m}}1\\
    &= N q^{2N}\Big(1-\frac{1}{q}\Big)+q^{2N},
\end{align*}
where the last line follows from \cref{lem:corteel}.
\end{proof}

\section{Combinatorial method in the large $q$ limit: proof of \cref{thm:q-limit}}
{\label{sec:combinatorial-general}}
This section is dedicated to the proof of \cref{thm:q-limit} by generalizing the combinatorial approach of \cref{sec:combinatorial-4}. As in the $k=2$ case, we can divide the equation $$F_1 \cdots F_k = G_1 \cdots G_k$$ by certain greatest common divisors to reduce the problem of computing the $2k$th moment to a problem of counting sets of polynomials with certain coprimality conditions. For $k>2$, these coprimality conditions are more complex and, to the best of our knowledge, there are no known exact formulas for counting the number of polynomials satisfying these conditions. For simplicity, rather than describing equivalent conditions to $F_1\cdots F_k = G_1 \cdots G_k$, we define an injection from \begin{equation}\label{set to count}\{(F_1, \dots, F_k, G_1, \dots, G_k) \in \cM_N^{2k}\ \big|\ F_1\cdots F_k = G_1 \cdots G_k\}\end{equation} to $\cM_{k^2}$ and then apply simple bounds to the image of \eqref{set to count}.

We first present the argument for $k=3$ as an example that will illustrate the general case. We consider
    $$F_1\,F_2\,F_3 = G_1\, G_2\, G_3,$$ where $F_i, G_i \in \cM_N$.
 
Due to the additional factor compared to the $k=2$ case, we will require multiple recursive steps to decompose the products. We begin by letting
\begin{align*}
    A_{12} &= (F_1, G_2)\\
    A_{23} &= (F_2, G_3)\\
    A_{31} &= (F_3, G_1).
\end{align*}
Here the indices are determined by the permutation $\sigma = (1\,2\,3)$, so that $A_{i\sigma(i)} = (F_i, G_{\sigma(i)})$. For the next set of GCDs, we repeat this shift in the indices by letting $$A_{i\,\sigma^2(i)} = \Big(\frac{F_i}{A_{i\,\sigma(i)}}, \frac{G_{\sigma^2(i)}}{A_{\sigma(i)\, \sigma^2(i)}}\Big),$$ so that
\begin{align*}
    A_{13} &= \Big(\frac{F_1}{A_{12}}, \frac{G_3}{A_{23}}\Big),\\
    A_{21} &= \Big(\frac{F_2}{A_{23}}, \frac{G_1}{A_{31}}\Big),\\
    A_{32} &= \Big(\frac{F_3}{A_{31}}, \frac{G_2}{A_{12}}\Big).\\
\end{align*}
Repeating the shifts a third time, we have
\begin{align*}
    A_{i\sigma(i)^3} = A_{ii} = \frac{F_i}{A_{i\sigma(i)} A_{i \sigma^2(i)}}= \frac{G_i}{A_{\sigma^{-1}(i) i} A_{ \sigma^{-2}(i) i}},
\end{align*}
with the last equality following from the fact that for 
 $i\neq j$, $$\Big(\frac{F_i}{A_{i\sigma(i)} A_{i \sigma^2(i)}}, \frac{G_j}{A_{\sigma^{-1}(j) j} A_{ \sigma^{-2}(j) j}}\Big) = 1.$$ 

We then have that \begin{align*}
    F_i = A_{i \sigma(i)} A_{i \sigma^2(i)} A_{ii} \qquad\text{and}\qquad G_j = A_{\sigma^{-1}(j) j} A_{ \sigma^{-2}(j) j} A_{jj}.
\end{align*}
Thus, we have mapped $(F_1, F_2, F_3, G_1, G_2, G_3)$ to a unique matrix $$
\begin{bmatrix}A_{11} &A_{12} & A_{13}\\
                A_{21} &A_{22} &A_{23}\\
                A_{31} &A_{32} &A_{33}\end{bmatrix}$$
in which the product of the entries in row $i$ is $F_i$ and the product of the entries in column $j$ is $G_j$. We will call this matrix the \emph{GCD matrix} of $(F_1, F_2, F_3, G_1, G_2, G_3)$. 

Now, letting $\deg A_{ij} = a_{ij}$, we note that the matrix containing the degrees of the polynomials in the GCD matrix, $$\begin{bmatrix}a_{11} &a_{12} & a_{13}\\
                a_{21} &a_{22} &a_{23}\\
                a_{31} &a_{32} &a_{33}\end{bmatrix},$$
is a magic square since row $i$ must sum to $\deg F_i = N$ and column $j$ must sum to $\deg G_j = N$. We will call this the \emph{GCD degree matrix} of $(F_1, F_2, F_3, G_1, G_2, G_3)$.

For $S$ a magic square of size $3$ for the number $N$, we let $R(S)$ count the number of tuples in $$\{(F_1, F_2, F_3, G_1, G_2, G_3) \in \cM_N^6\ \big|\ F_1 F_2 F_3 = G_1 G_2 G_3 \}$$  with GCD degree matrix given by $S$. We then have that 
$$\mathbb{E}\Big|\sum_{F\in \cM_N}f(F)\Big|^{6} = \sum_{S}R(S),$$  where the sum is taken over all magic squares of size $3$ for the number $N$. We would now like to bound $R(S)$. We first note that $R(S)$ is bounded above by the total number of $3\times3$ matrices with entries in $\cM$ and with the degrees of the entries given by $S$. Thus $R(S) \leq q^{a_{11}+a_{12}+\cdots+a_{32}+a_{33}} = q^{3N}$. To bound $R(S)$ from below, we consider a $3\times 3$ matrix with entries $A_{ij}$ in $\cM$ such that the entries are pairwise coprime and the degrees of the entries are given by $S$. We have that $(A_{i1} A_{i2} A_{i3},\, A_{1j}A_{2j}A_{3j}) = A_{ij}$, so this matrix is the GCD matrix of $(F_1, F_2, F_3, G_1, G_2, G_3)$, where $$F_i = A_{i1}A_{i2} A_{i3}$$ and $$G_j = A_{1j}A_{2j}A_{3j}.$$ It follows that $R(S)$ is bounded below by the number of pairwise coprime monic polynomials with degrees equal to the entries of $S$.

Before establishing an asymptotic for this lower bound, we return for the moment to the general case. 
For $k\geq 3$,
given $\sigma = (1\, 2\, \dots\, k)$ and $s\leq k$, we let $$A_{i\,\sigma^s(i)} = \Big(\frac{F_i}{\prod_{r<s}A_{i\,\sigma^r(i)}}, \frac{G_{\sigma^s(i)}}{\prod_{r<s}A_{\sigma^{-r}(i)\,\sigma^s(i)}}\Big).$$ As we saw for $k = 3$, we have
$$
A_{ii} = A_{i\sigma^k(i)} = \frac{F_i}{\prod_{r<k}A_{i\,\sigma^r(i)}} =  \frac{G_i}{\prod_{r<k}A_{\sigma^{-r}(i)\,i}},
$$
meaning that \begin{align*}
    F_i = \prod_{h = 1}^k A_{i\, h} \qquad \text{and}\qquad G_j = \prod_{h = 1}^k A_{h\,j}.
\end{align*}
In this way, we map $(F_1, \dots, F_k, G_1, \dots, G_k)$ to the GCD matrix $$
\begin{bmatrix}A_{11} &\cdots & A_{1k}\\
                \vdots &\ddots &\vdots\\
                A_{k1} &\cdots &A_{kk}\end{bmatrix}$$
in which the product of the entries in row $i$ is $F_i$ and the product of the entries in column $j$ is $G_j$. Letting $\deg A_{ij} = a_{ij}$, we note again that the GCD degree matrix, 
$$\begin{bmatrix}a_{11} &\cdots& a_{1k}\\
                \vdots &\ddots &\vdots\\
                a_{k1} &\cdots &a_{kk}\end{bmatrix},$$
is a magic square. We remark that a similar GCD matrix construction was explored by Granville and Soundararajan in \cite{GranvilleSound} in their study of $2k$th moments in the integer setting. 

For $S$ a magic square of size $k$ for the number $N$, let $R(S)$ count the number of tuples in $$\{(F_1, \dots, F_k, G_1, \dots, G_k) \in \cM_N^{2k}\ \big|\ F_1 \cdots F_k = G_1 \cdots G_k \}$$  with GCD degree matrix $S$, so that
$$\mathbb{E}\Big|\sum_{F\in \cM_N}f(F)\Big|^{2k} = \sum_{S}R(S).$$
As before, the sum is taken over all magic squares of size $k$ for the number $N$.
We observe that $R(S)$ is bounded above by the total number of $k\times k$ matrices with entries in $\cM$ and with the degrees of the entries given by $S$. Thus $R(S) \leq q^{\sum_{i=1}^k \sum_{j=1}^k a_{ij}} = q^{kN}$. To bound $R(S)$ from below, we consider a $k\times k$ matrix with pairwise coprime entries $A_{ij}$ in $\cM$ such that the degrees of the entries are given by $S$. We have that $$\Big(\prod_{h=1}^k A_{ih},\, \prod_{h=1}^k A_{hj}\Big) = A_{ij},$$ so this matrix is the GCD matrix of $(F_1, \dots, F_k, G_1, \dots, G_k)$, where \begin{align*}
    F_i = \prod_{h = 1}^k A_{i\, h} \qquad\text{and}\qquad G_j = \prod_{h = 1}^k A_{h\,j}.
\end{align*} It follows that $R(S)$ is bounded below by the number of pairwise coprime monic polynomials with degrees equal to the entries of $S$.

To obtain the $q$-limit asymptotic and complete the proof of \cref{thm:q-limit}, we will need the following result on the number of pairwise coprime polynomials of fixed degrees. 

\begin{lemma}[Corollary 3 in \cite{LiebProbabilityEstimates}]\label{lem:pairwise_coprime}
For $(d_1, \dots, d_s)\in \mathbb{N}^s$, the number of monic pairwise coprime polynomials $f_i \in \mathbb{F}_q[t]$  with $\deg(f_i)= d_i$ for $i=1, \dots, s$ is
$$ q^{d_1+\cdots+d_s}\left( 1-\frac{s(s-1)}{2}\cdot\frac{1}{q}+O_{s,d_1,\dots,d_s}\left(\frac{1}{q^2}\right)\right).$$
\end{lemma}

Applying \cref{lem:pairwise_coprime} to our lower bound for $R(S)$, we see that $$R(S) \geq q^{kN}\left( 1-\frac{k^2(k^2-1)}{2}\cdot\frac{1}{q}+O_{d_1,\dots,d_s}\left(\frac{1}{q^2}\right)\right).$$ Thus  $R(S) = q^{kN}\big(1 + O_{k,N}(q^{-1}) \big)$, and \cref{thm:q-limit} follows.

\section{Fourth moment using an analytic approach}
\label{sec:analytic-4}
We provide another proof of \cref{thm:k=2} by studying a power series in $4$ variables indexed over all $4$-tuples of monic polynomials:
\begin{equation}{\label{b_series}}
A(z_1,z_2,z_3,z_4) := \qsum_{F_1, \dots, F_{4}\in \mathcal{M}} \mathbbm{1}_{\{F_1F_2 = F_{3}F_{4}\}}z_1^{\deg F_1}\cdots z_{4}^{\deg F_{4}}. 
\end{equation}
We will check at the beginning of \cref{sec:analytic-general} in a more general case that $A$ converges uniformly on compact subsets of $\{ (z_1,z_2,z_3,z_4) \in \C^{4} : |z_i| < \tfrac{1}{q} \text{ for } 1 \leq i \leq 4\}$.  
By \eqref{eqn:setup}, calculating the fourth moment $\mathbb{E}| \sum_{F \in \mathcal{M}_N} f(F) |^4$  amounts to computing the coefficient of $z_1^N\cdots z_{4}^N$ in the above power series.

Observing that the indicator function $\mathbbm{1}_{\{F_1F_2 = F_{3}F_{4}\}}$ is multiplicative (see the \cref{sec:analytic-general} footnote for a definition of multiplicative), we find that

\begin{equation}{\label{b_bbig_ep}}
    A(z_1,z_2,z_3,z_4) = \prod_{P} \left(\qsum_{m_1+m_2 = m_{3}+m_4} (z_1^{m_1}z_2^{m_2}z_3^{m_3}z_{4}^{m_{4}})^{\deg P}\right).
\end{equation}
The product runs over all irreducible monic polynomials and the internal sum iterates over the possible ways to distribute a given irreducible factor $P$ among $F_1, \dots, F_{4}$ such that the number of factors of $P$ in $F_1F_2$ and $F_{3}F_{4}$ are equal. As in \cite{HeapLindqvist-2016}, we factor this product as
\begin{equation}\label{factorization_of_A}
A(z_1, \dots, z_{4}) = B(z_1, \dots, z_{4}) \prod_{i=1}^2 \prod_{j=3}^{4} \zeta_q(z_iz_j).
\end{equation}
Here, $B$ is defined implicitly by \eqref{factorization_of_A}, namely
\begin{equation}\label{B_k=2}
B(z_1, z_2, z_3, z_4) = \prod_{P}\Big(\qsum_{\substack{m_1+m_2=m_3+m_4}}\Big(z_1^{m_1}z_2^{m_2}z_3^{m_3}z_4^{m_4}\Big)^{\deg P}\Big)\prod_{i=1}^2\prod_{j=3}^4 \Big(1-(z_iz_j)^{\deg P}\Big).
\end{equation}

We note that absolute and uniform convergence of this expression for $|z_i| < \frac{1.1}{\sqrt{q}}$ will follow by applying \cref{prop:convergence_of_B}, which we prove in \cref{sec:analytic-B-convergence}, for $k =2$. 

Recall that by the multivariate version of Cauchy's integral formula, obtaining the desired coefficient of $A$ can be done by computing
\begin{align*}{\label{b_cauchy_integral}}
     &\frac{1}{(2\pi i)^{4}}\int_{|z_{4}|=\epsilon_{4}}\cdots \int_{|z_{1}|=\epsilon_1} \frac{A(z_1, \dots, z_{4})}{z_1^{N+1}\cdots z_{4}^{N+1}}dz_1\cdots dz_{4}\\
    &= \frac{1}{(2\pi i)^{4}}\int_{|z_{4}|=\epsilon_{4}}\cdots \int_{|z_{1}|=\epsilon_1} \frac{B(z_1, \dots, z_{4}) \prod_{i=1}^2 \prod_{j=3}^{4} \zeta_q(z_iz_j)}{z_1^{N+1}\cdots z_{4}^{N+1}}dz_1\cdots dz_{4}, 
\end{align*}
where we integrate over a polydisc with small radii $0 < \epsilon_i < q^{-1}$ around the origin. 
We proceed to find a simple formula for \eqref{B_k=2}
which will enable us to explicitly compute the above integral.
To do so, we study the local factors of the $B$-term in the following lemma.
\begin{lemma}
The following identity holds:
$$
\Big(\qsum_{m_1+m_2=m_3+m_4}z_1^{m_1}z_2^{m_2}z_3^{m_3}z_4^{m_4}\Big)\prod_{i=1}^2\prod_{j=3}^4 \left(1-z_iz_j\right) = 1-z_1z_2z_3z_4.
$$
\end{lemma}
\begin{proof}

Note that 
\begin{equation}\label{geometric_series_identity}
\prod_{i=1}^2 \prod_{j = 3}^4 \frac{1}{1-z_i z_j} = \qsum_{k_{13},k_{14},k_{23},k_{24}} z_1^{k_{13}+k_{14}} z_2^{k_{23}+k_{24}} z_3^{k_{13}+k_{23}} z_4^{k_{14}+k_{24}}.
\end{equation}
We can thus write the product of \eqref{geometric_series_identity} times $z_1z_2z_3z_4$ as
\begin{equation}\label{identity_multiplied_by_zs}
\qsum_{k_{13},k_{14},k_{23},k_{24}} z_1^{(k_{13}+1)+k_{14}} z_2^{k_{23}+(k_{24}+1)} z_3^{(k_{13}+1)+k_{23}} z_4^{k_{14}+(k_{24}+1)}.
\end{equation}
Subtracting this from \eqref{geometric_series_identity} amounts to multiplying \eqref{geometric_series_identity} by $(1-z_1 z_2 z_3 z_4)$. Now, notice that the monomial in \eqref{identity_multiplied_by_zs} associated to the index $(k_{13}-1,k_{14},k_{23},k_{24}-1)$ can be paired with the monomial associated to the index $(k_{13},k_{14},k_{23},k_{24})$ in \eqref{geometric_series_identity} to cancel each other out. This way of pairing terms is not unique, but with this choice the only remaining terms come from monomials in \eqref{geometric_series_identity} where either $k_{13} = 0$ or $k_{24} = 0$.

Now, it suffices to show that given a tuple $(m_1,m_2,m_3,m_4)$ such that 
$$
m_1 + m_2 = m_3 + m_4,
$$
there is a unique tuple $(k_{13},k_{14},k_{23},k_{24})$ with either $k_{13} = 0$ or $k_{24} = 0$ such that 
$$(m_1,m_2,m_3,m_4) = (k_{13}+k_{14},k_{23}+k_{24},k_{13}+k_{23},k_{14}+k_{24}).$$
To see this, assume without loss of generality that $m_1$ is maximal among the $m_i$. Then $m_1-m_4=k_{13}-k_{24}\geq 0$ so we must have $k_{24}=0$. It then follows that $k_{23}=m_2$, $k_{14}=m_4$, and $k_{13}=m_3-m_2$. 
\end{proof}

We then observe
\begin{align*}
    B(z_1, \dots, z_4) &= \prod_{P} \Big( 1-(z_1z_2z_3z_4)^{\deg P} \Big)\\
    &= \zeta_q(z_1z_2z_3z_4)^{-1}\\
    &= 1-qz_1z_2z_3z_4.
\end{align*}
Substituting our formula for $B(z_1, \dots, z_4)$,  the precise moment for $k=2$ is therefore 
$$\frac{1}{(2\pi i)^{4}}\int_{|z_{4}|=\epsilon_{4}}\cdots \int_{|z_{1}|=\epsilon_1} \frac{(1-qz_1z_2z_3z_4) \prod_{i=1}^2 \prod_{j=3}^{4} \zeta_q(z_iz_j)}{z_1^{N+1}\cdots z_{4}^{N+1}}dz_1\cdots dz_{4}, $$
which we may split using linearity as
\begin{align*}
\frac{1}{(2\pi i)^{4}}\int_{|z_{4}|=\epsilon_{4}}\cdots \int_{|z_{1}|=\epsilon_1} &\frac{\prod_{i=1}^2 \prod_{j=3}^{4} \zeta_q(z_iz_j)}{z_1^{N+1}\cdots z_{4}^{N+1}}dz_1\cdots dz_{4} \\
-&q \frac{1}{(2\pi i)^{4}}\int_{|z_{4}|=\epsilon_{4}}\cdots \int_{|z_{1}|=\epsilon_1} \frac{\prod_{i=1}^2 \prod_{j=3}^{4} \zeta_q(z_iz_j)}{z_1^{N}\cdots z_{4}^{N}}dz_1\cdots dz_{4}. 
\end{align*}

The following proposition will allow us to conclude. 
\begin{proposition} \label{prop:fourth-magic-square}
    Fix $0 <\epsilon < q^{-1}$. For each integer $k\geq 1$,
    \begin{equation}{\label{integral}}
        \frac{1}{(2\pi i)^{2k}}\int_{|z_{2k}|=\epsilon}\cdots \int_{|z_{1}|=\epsilon} \frac{\prod_{i=1}^k\prod_{j=k+1}^{2k} \zeta_q(z_iz_j)}{z_1^{N+1}\cdots z_{2k}^{N+1}}dz_1\cdots dz_{2k} = q^{Nk}\mathcal{S}_k(N).
    \end{equation}
\end{proposition}
Indeed, noting that $\mathcal{S}_2(N) = N+1$, it then follows that the fourth moment equals
$$(N+1)q^{2N} - q\cdot Nq^{2(N-1)} = Nq^{2N}\Big(1-\frac{1}{q}\Big) + q^{2N},$$
as desired.

\begin{proof}[Proof of \cref{prop:fourth-magic-square}]  Fix $j\in \{k+1,\dots, 2k\}$. Note that we can compute the integral over $z_j$ by computing the coefficient of $z_j^N$ in the product 
\begin{equation}\label{z_j product}\prod_{i=1}^{k} \zeta_q(z_iz_j).\end{equation}
Recall that $\zeta_q(z) = \sum_{n\geq 0}q^n\cdot z^n$. The coefficient of $z_j^N$ in \eqref{z_j product} is then just
$$q^N \msum_{m_{1}+\cdots +m_{k} = N} z_{1}^{m_{1}}\cdots z_{k}^{m_{k}}.$$
Observe that this product is independent of our choice of $j$, and since the $j$'s are independent of each  other in the product, we can integrate out $z_{k+1}, \dots, z_{2k}$. In this way we can transform the integral expression in \eqref{integral} into
$$\frac{1}{(2\pi i)^k}\int_{|z_{k}|=\epsilon}\dots \int_{|z_{1}|=\epsilon} \frac{q^{Nk}\Big(\ds\sum_{m_{1}+\cdots +m_{k} = N} z_{1}^{m_{1}}\cdots z_{k}^{m_{k}}\Big)^k}{z_{1}^{N+1} \cdots z_{k}^{N+1}}dz_{1}\cdots dz_{k}.$$

At this point we can directly compute the coefficient of $z_{1}^N\cdots z_{k}^N$ in the numerator.

Reindexing the sums, we expand it as
\begin{align*}
q^{Nk} \Big(\msum_{m_{1, 1}+\cdots +m_{k, 1} = N} z_{1}^{m_{1, 1}}\cdots z_{k}^{m_{k, 1}}\Big)  \cdots
\Big(\msum_{m_{1, k}+\cdots +m_{k, k} = N} z_{1}^{m_{1, k}}\cdots z_{k}^{m_{k, k}}\Big).
\end{align*}
This product is precisely
$$\msum_{\substack{m_{1, 1}+\cdots +m_{k, 1} = N\\ \vdots\\ m_{1, k}+\cdots +m_{k, k} = N}} z_{1}^{\sum_{j=1}^k m_{1, j}}\cdots z_{k}^{\sum_{j=1}^k m_{k, j}}.$$

We observe that the coefficient of $z_1^N\cdots z_k^N$ is given by the number of sets of $k^2$ integers $m_{i,j}$ such that in the square $$\begin{bmatrix}m_{1,1} &\cdots &m_{k,1}\\ \vdots &\ddots &\vdots\\ m_{1,k} &\cdots &m_{k,k}\end{bmatrix},$$ each row sums to $N$ (enforcing the condition over the summation) and each column sums to $N$ (ensuring that all exponents are raised to the $N$th power).
Thus the integral is precisely equal to $q^{Nk} \mathcal{S}_k(N)$.
\end{proof}

\begin{remark}
    Strictly speaking, the above analytic proof of \cref{thm:k=2} does not require any complex analysis or Cauchy's integral formula. The argument can be reformulated using only formal power series. We chose this approach because it closely parallels the proof of \cref{thm:steinhaus} and we hope the reader can benefit from those parallels.  
\end{remark}

\section{Analytic Computation of Steinhaus Moments: proof of \cref{thm:steinhaus}}
\label{sec:analytic-general}

The remainder of the paper will be dedicated to the proof of \cref{thm:steinhaus}. We provide a brief outline. This section establishes the theorem assuming some key propositions whose proof will be postponed to subsequent sections; the strategy closely parallels \cref{sec:analytic-4}. First, we form the formal power series $A(z_1,\dots,z_{2k})$ for counting solutions to \eqref{eqn:setup}. We express this series as a product of zeta functions and another power series, denoted $B(z_1,\dots,z_{2k})$, and then form the contour integral representation of \eqref{eqn:setup}. Second, using analytic properties of $B$  (\cref{prop:convergence_of_B} proved in \cref{sec:analytic-B-convergence}), we express \eqref{eqn:setup} as an iterated integral over an expanded polydisk after a suitable change of variables (\cref{prop:integral_representation_of_2kth_moment} proved in \cref{sec:analytic-integral-repn}).    Third, we extract the main contribution of this integral in terms of the arithmetic constant $b_k(q)$ and the number of magic squares $\mathcal{S}_k(N)$ (\cref{prop:magic_square_integral} proved in \cref{sec:analytic-magic-square}). Finally, we compute the size of the residual error term via a delicate analysis of an oscillatory integral (\cref{prop:error_term_estimation} proved in \cref{sec:analytic-error-integrals}), thus establishing the main theorem.  

\begin{proof}[Proof of \cref{thm:steinhaus}] One can check by hand that the explicit formula for $k=1$ matches the asymptotics given in \cref{thm:steinhaus}. We may thus assume $k \geq 2$. Furthermore, note that if $N$ eventually remains bounded by $20$ as $q^N \to \infty$, then \cref{thm:steinhaus} is already implied by \cref{thm:q-limit}. We may therefore also assume that $N$ is at least $20$. Both of these assumptions will be convenient for technical estimates later in the proof.  

We proceed as in \cref{sec:analytic-4}. Namely, we want to recover the coefficient of $z_1^N\cdots z_{2k}^N$ in the power series
\begin{equation} \label{eqn:GeneratingFunction}
    A(z_1,\dots, z_{2k}):=\msum_{F_1, \dots, F_{2k}\in \mathcal{M}} \mathbbm{1}_{\{F_1\cdots F_k = F_{k+1}\cdots F_{2k}\}}z_1^{\deg F_1}\cdots z_{2k}^{\deg F_{2k}}
\end{equation}
in $2k$ variables, which is precisely the $2k$th moment by \eqref{eqn:setup}. Note that this function is dominated by 
$$
\msum_{F_1,\dots,F_{2k} \in \mathcal{M}} |z_1^{\deg F_1} \cdots z_{2k}^{\deg F_{2k}}| = \prod_{i = 1}^{2k} \zeta_q(|z_i|),
$$
which converges uniformly on compact subsets of ${\{ (z_1,\dots,z_{2k}) \in \C^{2k} : |z_i| < \tfrac{1}{q} \text{ for } 1 \leq i \leq 2k\}}$.

Since the indicator function is multiplicative\footnote{For $r \in \mathbb{N}$, a  function $\varphi : \mathcal{M}^{r} \to \C$ is multiplicative if $\varphi(A_1B_1, \dots, A_rB_r) = \varphi(A_1,\dots,A_r) \varphi(B_1,\dots,B_r)$ for every $A_1,\dots,A_{r},B_1,\dots,B_r \in \mathcal{M}$ such that $\gcd(A_1\cdots A_r, B_1\cdots B_r) = 1$.}, we have the  product expansion 

\begin{equation*}
  A(z_1,\dots, z_{2k}) =   \prod_{P} \Bigg(\msum_{\substack{m_1+\cdots+m_k\\ = m_{k+1}+\cdots+m_{2k}}} (z_1^{m_1}\cdots z_{2k}^{m_{2k}})^{\deg P}\Bigg).
\end{equation*}

\begin{proposition}\label{prop:convergence_of_B}
Let 
$$
B(z_1,\dots,z_{2k}) := \prod_{P} \Bigg(\msum_{\substack{m_1+\cdots+m_k\\ = m_{k+1}+\cdots+m_{2k}}} (z_1^{m_1}\cdots z_{2k}^{m_{2k}})^{\deg P}\Bigg) \prod_{i = 1}^k \prod_{j = k+1}^{2k} (1-z_i^{\deg P}z_j^{\deg P}). 
$$
The infinite product $B(z)$ converges absolutely and uniformly on
$$\mathcal{R}:= \Big\{ (z_1,\dots,z_{2k}) \in \C^{2k} :  |z_i| <  \frac{1.1}{\sqrt{q}} \text{ for } 1 \leq i \leq 2k \Big\},$$
and hence $B(z)$ is holomorphic on $\mathcal{R}$. Furthermore, if $z,w\in \mathcal{R}$, then
$$
|B(z)| \ll_k 1 \quad \text{ and } \quad |B(z)-B(w)|\ll_{k} \frac{1}{\sqrt{q}} \|z-w\|_{\infty}.
$$
\end{proposition}

\begin{proof}
The proof is postponed to \cref{sec:analytic-B-convergence}.
\end{proof}

We may thus factor $A(z_1,\dots, z_{2k})$ as
$$A(z_1, \dots, z_{2k}) = B(z_1, \dots, z_{2k}) \prod_{i=1}^k \prod_{j=k+1}^{2k} \zeta_q(z_iz_j),$$
where this equation holds when both functions are defined.

By the multivariate version of Cauchy's integral formula, we have
$$ \mathbb{E} \Big | \sum_{F \in \cM_N} f(F) \Big|^{2k} = \frac{1}{(2\pi i)^{2k}}\int_{|z_{2k}|=\epsilon_{2k}}\cdots \int_{|z_{1}|=\epsilon_1} \frac{A(z_1, \dots, z_{2k})}{z_1^{N+1}\cdots z_{2k}^{N+1}} \, dz_1\cdots dz_{2k},$$
where $0 < \epsilon_i < q^{-1}$, so that $A$ converges on the closed polydisk with radii $\epsilon_i$.

\begin{proposition}\label{prop:integral_representation_of_2kth_moment} Let $c_i \in (1/4,3/4)$ for $i \in \{2,\dots,2k\}$.
We have the following equality:
\begin{align} \label{eqn:integral-repn}
\mathbb{E} \Big | \sum_{F \in \cM_N} f(F) \Big|^{2k} = \frac{N^{(k-1)^2} }{(2\pi)^{2k-1}}q^{kN}\int_{-\pi N}^{\pi N}\cdots \int_{-\pi N}^{\pi N} &\frac{e^{-(c_2+i\theta_2 + \cdots + c_{2k}+i\theta_{2k})}}{\prod_{i=2}^k\prod_{j=k+1}^{2k} N\left(1-e^{(c_i+i\theta_i + c_j+i\theta_j)/N}\right)}\\
\times &\prod_{j=k+1}^{2k}\frac{d\theta_j}{N\left(1-e^{(c_j+i\theta_j)/N}\right)} \nonumber \\
\times &B\left(q^{-1/2}, q^{-1/2}e^{(c_2 + i \theta_2)/N}, \dots, q^{-1/2}e^{(c_{2k}+i\theta_{2k})/N}\right) \nonumber \\
\times &\ d\theta_2\cdots d\theta_k. \nonumber
\end{align}
\end{proposition}

\begin{proof}
    The proof is postponed to \cref{sec:analytic-integral-repn}. The main idea is to use the factorization of $A$ in terms of $B$ and the zeta functions along with several substitutions to center the integrand around $z_1$. Finally, we plug in explicit parametrizations of the contour.
\end{proof}

\begin{remark}
    The integral \eqref{eqn:integral-repn} is similar to one considered by Heap--Lindqvist in Section 4.2 of \cite{HeapLindqvist-2016} when computing the moments of truncated characteristic polynomials with secular coefficients, and our proof proceeds similarly. For more discussion on connections with random matrix theory and secular coefficients, we refer the reader to works of Heap--Lindqvist \cite{HeapLindqvist-2016}, Conrey--Gamburd \cite{ConreyGamburd-2006},  Najnudel--Paquette--Simm \cite{NajnudelPaquetteSimm-2023}, and Gorodetsky \cite{Gorodetsky-2024}.   
\end{remark}

Now, we let $b_k(q) := B(q^{-1/2},\dots,q^{-1/2})$ so that $0 < b_k(q) \ll_k 1$ by \cref{prop:convergence_of_B}. Rewrite 
$$\mathbb{E} \Big | \sum_{F \in \cM_N} f(F) \Big|^{2k}=b_k(q) q^{kN}(I_k+J_k),$$
where
\begin{multline}
I_{k}= \frac{N^{(k-1)^2}}{(2\pi)^{2k-1}}\int_{-\pi N}^{\pi N} \cdots \int_{-\pi N}^{\pi N} \frac{e^{-(c_2+i\theta_2 + \cdots + c_{2k}+i\theta_{2k})}}{\prod_{i=2}^k\prod_{j=k+1}^{2k} N \left(1-e^{(c_i+i\theta_i + c_j+i\theta_j)/N}\right)} \\
\times \prod_{j=k+1}^{2k}\frac{d\theta_j}{N \left(1-e^{(c_j+i\theta_j)/N}\right)}\, d\theta_2\cdots d\theta_{2k} \nonumber
\end{multline}
and $$\displaystyle J_k= \frac{1}{b_k(q) q^{kN}} \mathbb{E} \Big | \sum_{F \in \cM_N} f(F) \Big|^{2k}-I_{k}.$$ 
We claim that $b_k(q) q^{kN} I_k$ is our main term.

\begin{proposition}\label{prop:magic_square_integral}
We have that $I_{k}= \mathcal{S}_k(N)$.
\end{proposition}

\begin{proof}
See \cref{sec:analytic-magic-square} for a detailed proof and Proposition 1 of \cite{ConreyGamburd-2006} for a similar statement. Our proof more closely follows part of the argument for Theorem 2 in \cite{ConreyGamburd-2006}. 
\end{proof}

It remains to show that the error term $J_k$ is $o(\mathcal{S}_k(N))$ as $qN\to\infty$. Since ${\mathcal{S}_k(N)/ N^{(k-1)^2} \ll_k 1}$ (see the remarks after \cref{thm:steinhaus}), it suffices to estimate the quantity
\begin{align}
\label{error_integral_I_prime}
   \widetilde{J_k} := \frac{1}{(2\pi)^{2k-1}}\int_{-\pi N}^{\pi N}&\cdots \int_{-\pi N}^{\pi N} \frac{e^{-(c_2+i\theta_2 + \cdots + c_{2k}+i\theta_{2k})}}{\prod_{i=2}^k\prod_{j=k+1}^{2k} N\left(1-e^{(c_i+i\theta_i + c_j+i\theta_j)/N}\right)}\\
\times &\prod_{j=k+1}^{2k}\frac{d\theta_j}{N\left(1-e^{(c_j+i\theta_j)/N}\right)} \nonumber \\
\times &\left( B(q^{-1/2}, q^{-1/2}e^{(c_2 + i \theta_2)/N}, \dots, q^{-1/2}e^{(c_{2k}+i\theta_{2k})/N}) - b_k(q) \right) \nonumber \\
\times &\ d\theta_2\cdots d\theta_{2k} \nonumber,
\end{align}
and show that $\widetilde{J_k}$ converges to 0 as $qN\to \infty$. 

\begin{proposition} \label{prop:error_term_estimation}
Fix $k \in \N$ with $k \geq 2$. For $2 \leq X \leq N/10$ and $10X \leq Y \leq N$ (recall the assumption $N \geq 20$), 
\begin{align*}
|\widetilde{J_k}| &\ll_k \frac{1}{q} \frac{(\log N)^{2k-1}}{X^{1-1/k}} + \frac{1}{q} (\log N)^{2k-2} \frac{X}{N} + \frac{1}{q} \Big( \frac{X}{Y} \Big)^{k-1} + \frac{1}{q} \frac{Y}{N} \left( \log N \right)^k.
\end{align*}
\end{proposition}

\begin{proof}
See \cref{sec:analytic-error-integrals} for the proof. The idea is to split up the region of integration.  The difference of $B$-terms will contribute the factor of $\frac{1}{q}$. We will show that if any of the $\theta_i$ or $\theta_j$ is large, i.e, not $o(N)$, the integral over that part of the region converges to $0$ as $N \rightarrow \infty$. On the other hand, if we have $\theta_i, \theta_j = o(N)$ for all $i,j$, the difference $B(\cdot)-b_k(q)$ will ensure convergence to $0$ even for fixed $q$.
\end{proof}

Choosing
$$
X = (N \log N)^{\frac{k}{2k-1}} 
\quad \text{ and } \quad 
Y = \frac{N^{1-1/k}}{X} = \frac{N^{\frac{k^2+k-1}{(2k-1)k}}}{(\log N)^{\frac{k}{2k-1}}},
$$
we obtain the final error estimate
\begin{align*}
\frac{1}{q} \Big( \frac{(\log N)^{4k^2-5k+2}}{N^{k-1}}\Big)^{\frac{1}{2k-1}} + \frac{1}{q} \Big( \frac{(\log N)^{2k^2-2k}}{N^{\frac{(k-1)^2}{k}}}\Big)^{\frac{1}{2k-1}} \ll_k \frac{1}{q}\frac{(\log N)^{k}}{N^{\frac{1}{2} - \frac{3k-2}{4k^2-2k}} }.
\end{align*}
A short computation shows that
\begin{align*}
    b_k(q) &= B(q^{-1/2}, \dots, q^{-1/2})\\
    &= \prod_{P} \sum_{m=0}^\infty \Big(\sum_{m_1+\cdots+m_k=m} \Big(\frac{1}{\sqrt{q}}\Big)^{m\deg P} \Big)^2 \prod_{i=1}^k \prod_{j=k+1}^{2k} \Big( 1-\Big(\frac{1}{q}\Big)^{\deg P}\Big).
\end{align*}
Using stars and bars for the inner sum, and noting that the second double product no longer depends on $i,j$, this reduces to
$$\prod_{P}\sum_{m=0}^\infty \Big(\binom{m+k-1}{k-1}^2 \Big(\frac{1}{\sqrt{q}}\Big)^{2m\deg P} \Big) \Big( 1-\Big(\frac{1}{q}\Big)^{\deg P}\Big)^{k^2}.$$
Thus,
$$b_k(q) = \prod_{P \in \cM} \Big( 1-\Big(\frac{1}{q}\Big)^{\deg P}\Big)^{k^2}\sum_{m=0}^\infty \binom{m+k-1}{k-1}^2 \Big(\frac{1}{q}\Big)^{m\deg P} .$$
This finishes the proof of \cref{thm:steinhaus} assuming \cref{prop:convergence_of_B,prop:integral_representation_of_2kth_moment,prop:magic_square_integral,prop:error_term_estimation}. 
\end{proof}
The last four sections will establish the remaining propositions. 

\section{Proof of \cref{prop:convergence_of_B}: Convergence of $B$ series}
\label{sec:analytic-B-convergence}

Recall that
$$
B(z_1,\dots,z_{2k}) := \prod_{P} \Big(\msum_{\substack{m_1+\cdots+m_k\\ = m_{k+1}+\cdots+m_{2k}}} (z_1^{m_1}\cdots z_{2k}^{m_{2k}})^{\deg P}\Big) \prod_{i = 1}^k \prod_{j = k+1}^{2k} (1-z_i^{\deg P}z_j^{\deg P}).
$$
Our proof proceeds by first looking at the local factors (i.e.\ for fixed $P$). Write $\ell := \deg P$ and let
$$
F_{\ell}(z_1,\dots,z_{2k}) := \msum_{\substack{m_1+\cdots+m_{k} \\ =m_{k+1}+\cdots+m_{2k}}}z_1^{m_1 \ell}\cdots z_{2k}^{m_{2k}\ell}
$$
and $$
G_{\ell}(z_1,\dots,z_{2k}) :=\prod_{i=1}^k\prod_{j=k+1}^{2k}(1-z_i^{\ell} z_j^{\ell}).
$$
Expanding the sums and products, we get
$$
F_{\ell}(z_1,\dots,z_{2k})=1+\sum_{i=1}^k\sum_{j=k+1}^{2k}z_i^{\ell} z_j^{\ell} + \sum_{M=2}^\infty \Big(\msum_{\substack{m_1+\cdots+m_k \\ =m_{k+1}+\cdots+m_{2k} \\ =M}} (z_1^{m_1}\cdots z_{2k}^{m_{2k}})^{\ell}\Big)
$$
and $$G_{\ell}(z_1,\dots,z_{2k})=1-\sum_{i=1}^k\sum_{j=k+1}^{2k}z_i^{\ell} z_j^{\ell} 
+ \sum_{L=2}^{k^2}\Big(\msum_{\substack{n_1+\cdots+n_k \\=n_{k+1}+\cdots+n_{2k}\\ =L \\ n_i\leq k}} c(n_1,\dots, n_{2k})(z_1^{n_1}\cdots z_{2k}^{n_{2k}})^{\ell}\Big)
$$
for some constants $c(n_1,\dots,n_{2k})$ depending on $n_1,\dots,n_{2k}$.
Consequently,
$$
(F_{\ell}\cdot G_{\ell})(z_1,\dots,z_{2k}) = 1 + \dsum_{\substack{M, L\\ M+L\geq 2\\ L\leq k^2}}\Big(\msum_{\substack{m_1' + \cdots + m_{k}' \\ = m_{k+1}' + \cdots +m_{2k}' \\ = M+L}} c'(m_1', \dots, m_{2k}') (z_1^{m_1'}\cdots z_{2k}^{m_{2k}'})^{\ell}\Big)
$$
for some other constants $c'(m_1',\dots,m_{2k}')$. These constants are uniformly bounded in $k$. This local factor of $B$ can be rewritten as 
$
1+ C_P(z),
$
where 
\begin{equation} \label{eq:cpz}
    C_P(z)= \sum_{M=2}^\infty \msum_{\substack{m_1+\cdots+m_k \\ =m_{k+1}+\cdots+m_{2k} \\ =M}} c'(m_1,\dots,m_{2k}) z_1^{m_1 \ell}\cdots z_{2k}^{m_{2k} \ell}.
\end{equation}
For simplicity, fix $D = 1.21  e^{0.1} > 1$ (this choice is somewhat arbitrary). We now give bounds on the local factor.

\begin{lemma}{\label{lem:c_estimates}}
For $z\in\mathcal{R} = \{ (z_1,\dots,z_{2k}) \in \C^{2k} : |z_i| < 1.1 q^{-1/2} \text{ for } 1 \leq i \leq 2k \}$, we have
$$
|C_P(z)| \ll_k {\Big( \frac{D}{q} \Big)}^{2 \ell}
$$
and
$$
|C_P(z)- C_P(w)| \ll_k \ell ||z-w||_{\infty} \sqrt{q} {\Big( \frac{D}{q} \Big)}^{2\ell}.
$$
\end{lemma}

\begin{proof}
For $z\in\mathcal{R}$, we have
\begin{align} \label{eq:cp}
    |C_P(z)|&\ll_k \sum_{M=2}^\infty \Big(\msum_{\substack{m_1+\cdots+m_k \\ =m_{k+1}+\cdots+m_{2k} \\ =M}}1\Big)\Big(\frac{1.1}{\sqrt{q}}\Big)^{2M\ell} \nonumber \\
    &= \sum_{M=2}^\infty \binom{M+k-1}{k-1}^2 \Big(\frac{(1.1)^2}{q}\Big)^{M\ell} \\
    &\ll_k \sum_{M=2}^{\infty} M^{2(k-1)} \Big(\frac{1.21}{q}\Big)^{M\ell} \nonumber
\end{align}
where the last inequality follows since $\binom{M+k-1}{k-1}^2$ is a polynomial in $M$ of degree at most $2(k-1)$. Observe that $\big(\frac{1.21}{q}\big)^{M\ell} = e^{-0.1M\ell}\big(\frac{D}{q}\big)^{M\ell}$, and further for $M \geq 2$, 
$$
M^{2(k-1)}e^{-0.1M\ell} \leq M^{2(k-1)}e^{-0.1M} \ll_k 1.
$$
We conclude that 
\begin{align}{\label{eq:single_c}}
     \sum_{M=2}^{\infty} M^{2(k-1)} \Big(\frac{1.21}{q}\Big)^{M\ell} &\ll_k \sum_{M=2}^{\infty} \Big(\frac{D}{q}\Big)^{M\ell} \ll_k \Big( \frac{D}{q}\Big)^{2\ell},
\end{align}
where the last estimate comes from bounding the convergent geometric series by its first term. This proves the first estimate in the lemma.

Now, we prove the upper bound for the difference $|C_P(z)-C_P(w)|$. From the definition in \eqref{eq:cpz} and using that $|z_i|,|w_i|< 1.1 q^{-1/2}$, we find that
\begin{align} \label{eq:cpdiff}
|C_P(z)-C_P(w)|
&\ll_k
\sum_{M=2}^\infty \msum_{\substack{m_1+\cdots+m_k \\ =m_{k+1}+\cdots+m_{2k} \\ =M}} \big|(z_1^{m_1 }\cdots z_{2k}^{m_{2k}})^{\ell}  - (w_1^{m_1 }\cdots w_{2k}^{m_{2k}})^{\ell} \big|\nonumber \\
&\leq \ell \sum_{M=2}^\infty q^{-M(\ell-1)}(1.1)^{2M(\ell-1)} \msum_{\substack{m_1+\cdots+m_k \\ =m_{k+1}+\cdots+m_{2k} \\ =M}} \big|z_1^{m_1 }\cdots z_{2k}^{m_{2k}}  - w_1^{m_1}\cdots w_{2k}^{m_{2k}}\big|,
 \end{align}
 where we implicitly used the identity $A^n-B^n=(A-B)\sum_{k=0}^{n-1}A^{n-1-k}B^k$.

Now, we claim 
\begin{equation} \label{eq:zwd}
|z_1^{m_1 }\cdots z_{2k}^{m_{2k}}  - w_1^{m_1}\cdots w_{2k}^{m_{2k}}| < 2M \Big(\frac{1.1}{\sqrt{q}}\Big)^{2M-1}\|z-w\|_\infty.
\end{equation}
To see why, observe that when $k=1$, we can write
\begin{align*}
|z_1^{m_1}z_2^{m_2}-w_1^{m_1}w_2^{m_2}| &= |z_1^{m_1}z_2^{m_2}-z_1^{m_1}w_2^{m_2}+z_1^{m_1}w_2^{m_2}-w_1^{m_1}w_2^{m_2}| \\
&\leq |z_1|^{m_1}|z_2^{m_2}-w_2^{m_2}| + |w_2^{m_2}||z_1^{m_1}-w_1^{m_1}| \\
& < \Big(\frac{1.1}{\sqrt{q}}\Big)^{m_1}|z_2-w_2|m_2\Big(\frac{1.1}{\sqrt{q}}\Big)^{m_2-1} + \Big(\frac{1.1}{\sqrt{q}}\Big)^{m_2}|z_1-w_1|m_1\Big(\frac{1.1}{\sqrt{q}}\Big)^{m_1-1} \\
&\leq \|z-w\|_{\infty}(m_1+m_2) \Big(\frac{1.1}{\sqrt{q}}\Big)^{m_1+m_2-1}.
\end{align*}
Then, for any other $k\geq 1$, the claim follows by strong induction after rewriting 
\begin{multline*}
    z_1^{m_1 }\cdots z_{2k}^{m_{2k}}  - w_1^{m_1}\cdots w_{2k}^{m_{2k}} = z_1^{m_1 }z_3^{m_3}\cdots z_{2k-1}^{m_{2k-1}}(z_2^{m_2}z_4^{m_4}\cdots z_{2k}^{m_{2k}} - w_2^{m_2}w_4^{m_4}\cdots w_{2k}^{m_{2k}}) \\
    + w_2^{m_2}w_4^{m_4}\cdots w_{2k}^{m_{2k}}(z_1^{m_1}z_3^{m_3}\cdots z_{2k-1}^{m_{2k-1}} - w_1^{m_1}w_3^{m_3}\cdots w_{2k-1}^{m_{2k-1}}).
\end{multline*}

Thus, the estimate \eqref{eq:zwd} implies that \eqref{eq:cpdiff} is 
$$
\ll \ell\, \|z-w\|_\infty \sqrt{q}\sum_{M=2}^\infty M \binom{M+k-1}{k-1}^2 \Big(\frac{1.21}{q}\Big)^{M\ell}.
$$
At this point, we use the same manipulation as we did in going from \eqref{eq:cp} to \eqref{eq:single_c} to recover the estimate
$$
|C_P(z)-C_P(w)| \ll_k \ell ||z-w||_{\infty} \sqrt{q}\left(\frac{D}{q}\right)^{2\ell}.
$$
\end{proof}

Recall  by definition, $B(z) = \prod_{P } (1+C_P(z))$. By \cref{lem:c_estimates} and the observation ${D < 2 \leq q}$, if $\ell_0  \geq 1$ is sufficiently large depending only on $k$, then
\begin{equation} \label{eqn:ell_0}
|C_P(z)| \leq \frac{1}{2} \quad \text{whenever } \deg P \geq \ell_0. 
\end{equation}
For $\ell_0 \geq 1$, define
$$J(z) := \prod_{\substack{P \\ \deg P < \ell_0}}\big(1+C_P(z)\big), \qquad K(z) := \prod_{\substack{P \\ \deg P \geq \ell_0}}\big(1+C_P(z)\big)$$
so that $B(z) = J(z)K(z)$.

\begin{lemma}{\label{lem:jk_bounds}}
   Assume $\ell_0 = \ell_0(k)$ depends only on $k$ and  satisfies \eqref{eqn:ell_0}. For $z \in \mathcal{R}$,
    $$
    J(z) \ll_{k} 1, \qquad 
    K(z) \ll_{k} 1.
    $$
    In particular, $|B(z)| \ll_k 1$. 
\end{lemma}
\begin{proof}
Recall the prime polynomial theorem implies that 
\begin{equation} \label{eqn:PPT}
\pi_q(\ell) := \#\{ P \in \mathcal{M}_{\ell} : P \text{ irreducible} \} \ll \frac{q^{\ell}}{\ell},
\end{equation}
so we have
    \begin{align*}
        |J(z)| &= \Big|\prod_{\substack{P \\ \deg P < \ell_0}} (1 + C_P(z)) \Big| \\
        &\ll_k \prod_{\substack{P \\ \deg P < \ell_0}} \Big(1 + \frac{D^{2 \deg P}}{q^{2 \deg P}}\Big) = \prod_{1 \leq \ell < \ell_0} \Big(1 + \frac{D^{2 \ell}}{q^{2l}}\Big)^{\pi_q(\ell)} \\
        &\leq \prod_{1 \leq \ell < \ell_0} \exp\Big( O(q^{-\ell}D^{2\ell}/\ell) \Big) \ll_k 1.
    \end{align*}
Here, we have used that $(1+1/x)^x < e$ for all $x > 0$.

By our assumption on $\ell_0$, we have $|C_P(z)| \leq 1/2$, so we may use the expansion ${|\log(1+w)| = O(|w|)}$ for $|w| \leq 1/2$. Thus, the prime polynomial theorem \eqref{eqn:PPT} and \cref{lem:c_estimates} imply  that 
\begin{align*}
|\log(K(z))| &\ll \sum_{\substack{P \\ \deg P \geq \ell_0}} |C_P(z)| \ll_k \sum_{\substack{P \\ \deg P \geq \ell_0}} \frac{D^{2 \ell}}{q^{2 \ell}} \ll_k \sum_{\ell \geq \ell_0} \frac{D^{2\ell}}{q^{2\ell}} \frac{q^{\ell}}{\ell} \ll_k q^{-\ell_0} \ll_k 1.
\end{align*}
In particular, $|K(z)| \ll_{k} 1$.
\end{proof}

 \cref{lem:jk_bounds} shows in particular that $|B(z)| \ll_k 1$ as claimed in \cref{prop:convergence_of_B}. The remainder of \cref{prop:convergence_of_B} rests on the following consequence of the triangle inequality: for $z,w \in \mathcal{R}$, 
\begin{align}\label{eq:bounding_differences_of_B_terms_with_j_k}
    |B(z)-B(w)| &= |J(z)K(z) - J(w)K(w)| \nonumber\\
    &\leq |J(z)-J(w)| \cdot |K(z)| + | K(z) - K(w) | \cdot | J(w)|.
\end{align}
Therefore, showing that the differences $|J(z)-J(w)|$ and $|K(w)-K(z)|$ are bounded in terms of $k$ and explicitly in terms of $q$ will allow us to complete the proof of \cref{prop:convergence_of_B}.

We now choose a suitable value for $\ell_0$ by means of the next lemma.

\begin{lemma}\label{lem:choosing_ell}
   	There exists $\ell_0 = \ell_0(k)$ sufficiently large such that \eqref{eqn:ell_0} holds and for $z,w \in \mathcal{R}$, 
    \[
      \Big|\log \frac{K(z)}{K(w)} \Big| \leq \dfrac{1}{2} \quad \text{ and } \quad   \Big|\log \frac{K(z)}{K(w)} \Big| \ll_k \sqrt{q} ||z-w||_\infty \Big(\frac{D^2}{\sqrt{q}}\Big)^{\ell_0}. 
    \]
 \end{lemma}

\begin{proof}
    By \cref{lem:c_estimates},  $K(z)$ is an absolutely convergent infinite product inside $\mathcal{R}$. So, it is never zero in that region and we can take its logarithm. Using the Taylor expansion of $\log(1+z)$ for $|z| \leq 1/2$, we find that
    \begin{align*} \label{eq:logkk}
    \Big|\log \frac{K(z)}{K(w)}\Big| &= \Big|\sum_{\substack{P \\ \deg P \geq \ell_0}} \big(\log(1+C_P(z))-\log(1+C_P(w))\big)\Big| \\ 
    &= \Big| \sum_{\substack{P \\ \deg P \geq \ell_0}} \sum_{n = 1}^\infty (-1)^{n+1} \frac{C_P(z)^n-C_P(w)^n}{n} \Big| \\
    &\leq \sum_{\substack{P \\ \deg P \geq \ell_0}} \sum_{n = 1}^\infty \frac{|C_P(z) - C_p(w)|}{n} \sum_{i = 0}^{n-1} |C_P(z)^i C_P(w)^{n-1-i}| \\
    &\leq \sum_{\substack{P \\ \deg P \geq \ell_0}} |C_P(z) - C_p(w)|.
    \end{align*}

    Applying \cref{lem:c_estimates} as well as the prime polynomial theorem \eqref{eqn:PPT}, we see that
    \begin{align*}
        \sum_{\substack{P \\ \deg P \geq \ell_0}} \big(C_P(z)-C_P(w)\big) &\ll_k \sum_{\substack{\ell \geq \ell_0}} \ell ||z-w||_{\infty} \sqrt{q} \Big( \frac{D}{q}\Big)^{2\ell}\frac{q^{\ell}}{\ell}\\
        &=  \sqrt{q} ||z-w||_{\infty}\sum_{\ell \geq \ell_0} \Big(\frac{D^2}{q}\Big)^{\ell}\\
        &\ll_k \sqrt{q}||z-w||_\infty \Big(\frac{D^2}{q}\Big)^{\ell_0}.
    \end{align*}
    This is the second estimate that we claimed. Since $||z-w||_\infty \ll 1/\sqrt{q}$ and $D^2/q < 1$ for any $q \geq 2$, the first estimate also holds provided $\ell_0 = \ell_0(k)$ is taken to be sufficiently large.
\end{proof}

We proceed to bound the differences in \eqref{eq:bounding_differences_of_B_terms_with_j_k}, starting with $|K(z)- K(w)|$.

\begin{lemma} \label{lem:difofk}
   If $\ell_0 = \ell_0(k)$ is sufficiently large  according to \cref{lem:choosing_ell}, then
    $$
    |K(z)- K(w)| \ll_k \sqrt{q}  ||z-w||_{\infty} {\Big(\frac{D^2}{q}\Big)}^{\ell_0}.
    $$
\end{lemma}

\begin{proof}

We simply observe that
$$
|K(z)-K(w)| \ll_k \Big|\frac{K(z)}{K(w)}-1\Big| \ll \Big|\log \frac{K(z)}{K(w)}\Big|,
$$
where we used \cref{lem:jk_bounds} for the first bound, and $e^z = 1 + O(|z|)$ for $|z| \leq 1/2$ and \cref{lem:choosing_ell} for the second bound.
\end{proof}

The last estimate is an upper bound for the difference $|J(z)-J(w)|$. 

\begin{lemma} \label{lem:j_difference} For $z,w \in \mathcal{R}$, we have that
    $$
    |J(z) - J(w)| \ll_{k} \ell_0 ||z-w||_\infty \frac{D^2}{\sqrt{q}}.
    $$
\end{lemma} 

\begin{proof}
Let $\mathcal{N}$ be some subset of the set of all monic irreducible polynomials $P$ of degree less than $\ell_0$. We will show by induction on $|\mathcal{N}|$ that 
$$
\Big|\prod_{P \in \mathcal{N}} (1+C_P(z)) - \prod_{P \in \mathcal{N}} (1+C_P(w))\Big| \ll_{k} \ell_0 \sqrt{q} \sum_{P \in \mathcal{N}} ||z-w||_\infty \Big(\frac{D}{q}\Big)^{2 \deg P}.
$$
If $|\mathcal{N}| = 1$, this statement is just \cref{lem:c_estimates} and the fact that $\deg P < \ell_0$.
In general, we have
\begin{align*}
    \Big|\prod_{P \in \mathcal{N}}& (1+C_P(z)) - \prod_{P \in \mathcal{N}} (1+C_P(w))\Big| \\
    &\leq |C_{P_1}(z) - C_{P_1}(w)| \prod_{\substack{P \in \mathcal{N} \setminus \{P_1\}}} \left| 1 + C_P(z) \right| \\
    &+
    \Big|
    (1 + C_{P_1}(w)) \Big( \prod_{\substack{P \in \mathcal{N} \setminus \{P_1\}}} (1 + C_P(z)) - \prod_{\substack{P \in \mathcal{N} \setminus \{P_1\}}} (1 + C_P(w))\Big)
    \Big|
\end{align*}
for any $P_1 \in \mathcal{N}$. Now, we can trivially bound $|1+C_{P_1}(w)| \ll_{k} 1$, and by essentially the same argument as in \cref{lem:jk_bounds}, $\left|\prod_{\substack{P \in \mathcal{N} \setminus \{P_1\}}} (1 + C_P(z))\right| \ll_{k} 1$. Applying the induction hypothesis and \cref{lem:c_estimates}, we find that the above is bounded by
$$
\ell_0\sqrt{q}\,||z-w||_\infty \Big(\frac{D}{q}\Big)^{2 \deg P_1} + \ell_0\sum_{P \in \mathcal{N} \setminus \{P_1\}} \sqrt{q}\,||z-w||_\infty \Big(\frac{D}{q}\Big)^{2 \deg P},
$$
up to constants depending only on $k$, as claimed. In particular, we now get
\begin{align*}
    |J(z) - J(w)|
    &\ll_{k} \ell_0 \sqrt{q} \sum_{\substack{P \\ \deg P < \ell_0}} ||z-w||_\infty {\Big(\frac{D}{q}\Big)}^{2 \deg P}
    \\
    &\ll \ell_0\sqrt{q} \sum_{\ell < \ell_0} ||z-w||_\infty {\Big(\frac{D}{q}\Big)}^{2 \ell} \frac{q^{\ell}}{\ell} \ll \ell_0 ||z-w||_\infty \frac{D^2}{\sqrt{q}},
\end{align*}
by the prime polynomial theorem \eqref{eqn:PPT}. This proves the lemma.
\end{proof}

Inserting \cref{lem:j_difference,lem:difofk,lem:jk_bounds} in \eqref{eq:bounding_differences_of_B_terms_with_j_k}, we obtain
\begin{align*}
    |B(z)-B(w)| &\ll_k ||z-w||_{\infty}\Big( \frac{2}{\sqrt{q}}\Big) + ||z-w||_{\infty}\sqrt{q}\Big(\frac{2}{q}\Big)^{\ell_0} \ll_k   \frac{1}{\sqrt{q}} ||z-w||_{\infty},
\end{align*} 
as desired. This completes the proof of \cref{prop:convergence_of_B}. 

\section{Proof of \cref{prop:integral_representation_of_2kth_moment}: Integral representation}
\label{sec:analytic-integral-repn}

Recall that
$$ \mathbb{E} \Big | \sum_{F \in \cM_N} f(F) \Big|^{2k} = \frac{1}{(2\pi i)^{2k}}\int_{|z_{2k}|=\epsilon_{2k}}\cdots \int_{|z_{1}|=\epsilon_1} \frac{A(z_1, \dots, z_{2k})}{z_1^{N+1}\cdots z_{2k}^{N+1}} \, dz_1\cdots dz_{2k}.$$
Recalling that $\zeta_q(z)=\frac{1}{1-qz}$, we can further rewrite this as
\begin{equation}{\label{integral_with_B}}
    \frac{1}{(2\pi i)^{2k}} \int_{|z_{2k}| = \epsilon_{2k}} \cdots\int_{|z_1|=\epsilon_1} \frac{(z_1\cdots z_{2k})^{-(N+1)}}{\prod_{i=1}^k\prod_{j=k+1}^{2k}(1-qz_iz_j)}\cdot B(z_1, \dots, z_{2k}) \, dz_1\cdots dz_{2k}.
\end{equation}
We make several changes of variable as a means of centering the integral around $z_1$. In each case the contour integrals can still be taken along circles of small radii around $0$. First, we modify $z_{k+1}$ through $z_{2k}$ by sending
$$z_j \mapsto z_1^{-1}q^{-1}z_j \quad \text{ for }\quad k+1\leq j\leq 2k.$$
This yields the expression
$$\frac{q^{kN}}{(2\pi i)^{2k}} \int \cdots\int \frac{(z_1^{1-k} z_2 \cdots z_{2k})^{-(N+1)}}{\prod_{i=2}^k\prod_{j=k+1}^{2k}(1-z_iz_j/z_1)}\cdot \frac{B\left(z_1, \dots, z_1^{-1}q^{-1}z_{2k}\right)}{z_1^k\prod_{j=k+1}^{2k} (1-z_j)} \, dz_1\cdots dz_{2k}.$$
Here, the second product in the denominator of the integrand corresponds to the contribution from the double product in \eqref{integral_with_B} in the case where $z_i = z_1$. We may also factor out a $q^{kN}$ by considering the contributions from the numerator of the integrand in \eqref{integral_with_B} as well as the changes in the increment.

We make the second substitution
$$z_i \mapsto z_1z_i \quad \text{ for }\quad 2\leq i\leq k,$$
and our integral becomes
\begin{equation}{\label{post_COV}}
    \frac{q^{kN}}{(2\pi i)^{2k}} \int \cdots\int \frac{1}{z_1}\frac{(z_2 \cdots z_{2k})^{-(N+1)}}{\prod_{i=2}^k\prod_{j=k+1}^{2k}(1-z_iz_j)}\cdot \frac{B\left(z_1, z_1z_2,\dots, z_1^{-1}q^{-1}z_{2k}\right)}{\prod_{j=k+1}^{2k} (1-z_j)} \, dz_1\cdots dz_{2k}.
\end{equation}

Before proceeding with our analysis of \eqref{post_COV}, we observe that $B(z_1, z_1z_2, \dots, z_1^{-1}q^{-1}z_{2k})$ is independent of the value of $z_1$. Indeed,
\begin{align*}
    B(z_1, z_1z_2, \dots, z_1^{-1}q^{-1}z_{2k}) = \prod_{P} & \msum_{\substack{m_1+\cdots+m_k=\\ m_{k+1}+\cdots+m_{2k}}} \Big(\frac{z_1^{m_1+\cdots+m_k}}{z_1^{m_{k+1}+\cdots+m_{2k}}}\Big)^{\deg P}(z_2^{m_2}\cdots q^{-m_{2k}}z_{2k}^{m_{2k}})^{\deg P}\\
    \times &\prod_{i=2}^k \prod_{j=k+1}^{2k} \Big( 1-(q^{-1}z_iz_j)^{\deg P}\Big)\prod_{j=k+1}^{2k} \Big( 1-(q^{-1}z_j)^{\deg P} \Big).
\end{align*}
Since $m_1+\cdots +m_k = m_{k+1} + \cdots +m_{2k}$, the $z_1$ term vanishes and we are left with
$$
\prod_{P}  \msum_{\substack{m_1+\cdots+m_k\\ = m_{k+1}+\cdots+m_{2k}}}(z_2^{m_2}\cdots q^{-m_{2k}}z_{2k}^{m_{2k}})^{\deg P}\prod_{i=2}^k \prod_{j=k+1}^{2k} (1-(q^{-1}z_iz_j)^{\deg P}) \prod_{j=k+1}^{2k} (1-(q^{-1}z_j)^{\deg P}), 
$$
which is clearly independent of $z_1$. We further note that in \eqref{post_COV}, the only term that depends on $z_1$ is the $1/z_1$ factor. Thus, we can integrate out $z_1$ multiplying our expression by $2\pi i$. 

In the input of $B$, we can thus pick any value for $z_1$, so without loss of generality take $z_1=q^{-1/2}$. Our integral now becomes
\begin{equation}{\label{no_z1}}
    \frac{q^{kN}}{(2\pi i)^{2k-1}}\int\cdots \int \frac{(z_2 \cdots z_{2k})^{-(N+1)}}{\prod_{i=2}^k\prod_{j=k+1}^{2k}(1-z_iz_j)}\cdot \frac{B\left(q^{-1/2}, q^{-1/2}z_2\cdots, q^{-1/2}z_{2k}\right)}{\prod_{j=k+1}^{2k} (1-z_j)} \, dz_2\cdots dz_{2k}.
\end{equation}
Despite the changes of variables, the $i$th integral is still taken around some disk of positive radius $\epsilon_i$ which now only have to satisfy $0 < \epsilon_i < 1.1$ due to \cref{prop:convergence_of_B} and our choice $z_1=q^{-1/2}$.  For each integral, we can parametrize its contour in terms of its argument $\theta_i$ by
$$
\gamma(\theta_i) = e^{c_i/N} e^{i\theta_i}.
$$
Here, the $c_i \in (1/2,3/4)$ are some constants, chosen to be independent of $N$ and $q$. Note the radius satisfies $e^{c_i/N} \leq e^{3/80} < 1.1$ as $N \geq 20$ by assumption. Explicitly substituting the contours, we rewrite \eqref{no_z1} as
\begin{align}{\label{parametrized}}
    \frac{q^{kN}}{(2\pi)^{2k-1}} \int_{-\pi}^\pi \cdots\int_{-\pi}^\pi &\frac{e^{-c_2-\cdots-c_{2k}-iN(\theta_2+\cdots + \theta_{2k})}}{\prod_{i=2}^k\prod_{j=k+1}^{2k}\left(1-e^{(c_i+c_j)/N+i\theta_i+i\theta_j}\right)}\cdot \prod_{j=k+1}^{2k} \frac{1}{1-e^{c_j/N+i\theta_j}}\\
    &\times B\left(q^{-1/2}, q^{-1/2}e^{c_2/N + i\theta_2}, \dots, q^{-1/2}e^{c_{2k}/N+i\theta_{2k}}\right)d\theta_2\cdots d{\theta_{2k}}. \nonumber
\end{align}
Here, we see some cancellations arising from the derivative term and the original numerator of the integrand from \eqref{no_z1}. Note that the factor of $i^{2k-1}$ cancels with the constant term preceding the integral. At this point, we make the  change of variables
\begin{align*}
&\theta_j\mapsto \theta_j/N \quad \text{ for }\quad j=2,3,\dots,2k.
\end{align*}
We will also multiply the numerator and denominator by $N^{(k-1)^2}$, preceding each term in each product of the denominator in \eqref{parametrized} by $N$. This ultimately yields
\begin{align*}
\frac{N^{(k-1)^2} }{(2\pi)^{2k-1}}q^{kN}\int_{-\pi N}^{\pi N}\cdots \int_{-\pi N}^{\pi N} &\frac{e^{-(c_2+i\theta_2 + \cdots + c_{2k}+i\theta_{2k})}}{\prod_{i=2}^k\prod_{j=k+1}^{2k} N\left(1-e^{(c_i+i\theta_i + c_j+i\theta_j)/N}\right)}\\
\times &\prod_{j=k+1}^{2k}\frac{d\theta_j}{N\left(1-e^{(c_j+i\theta_j)/N}\right)} \nonumber \\
\times &B\left(q^{-1/2}, q^{-1/2}e^{(c_2 + i \theta_2)/N}, \dots, q^{-1/2}e^{(c_{2k}+i\theta_{2k})/N}\right) \nonumber \\
\times &\ d\theta_2\cdots d\theta_k, \nonumber
\end{align*}
completing the proof of \cref{prop:integral_representation_of_2kth_moment}. 

\section{Proof of \cref{prop:magic_square_integral}: magic squares}

\label{sec:analytic-magic-square}
This section is nearly identical to \cite[Section 3]{ConreyGamburd-2006} but we have included the argument for the sake of completeness. Making the change of variables $\theta_j \mapsto -\theta_j$ for $j=2,\dots,2k$, we turn each integral into a complex line integral along some vertical strip of length $2\pi N$ with real part $-c_j$. Letting $s_j = -c_j + i \theta_j$, our integral becomes
\begin{align*}
I_k = \frac{N^{1-2k}}{(2\pi i)^{2k-1}}\int_{-c_{2k} - i \pi N }^{ -c_{2k} + i \pi N}\cdots \int_{-c_{2} - i\pi N}^{-c_2 + i \pi N} &\frac{e^{s_2+\cdots+s_{2k}}}{\prod_{i=2}^k\prod_{j=k+1}^{2k} \left(1-e^{-(s_i+s_j)/N}\right)}\\
\times &\prod_{j=k+1}^{2k}\frac{ds_j}{\left(1-e^{-s_j/N}\right)} \nonumber \\
\times &\ ds_2\cdots ds_{2k}. \nonumber
\end{align*}
Let $u := (s_2,\dots,s_k)$ and $v := (s_{k+1}, \dots, s_{2k})$. Then, using the geometric series formula, the integrand becomes 
\begin{align*}
 &e^{s_1+\cdots+s_{2k}} \prod_{i=2}^k\prod_{j=k+1}^{2k} \frac{1}{\Big(1-e^{-(s_i+s_j)/N}\Big)}
 \times \prod_{j=k+1}^{2k}\frac{1}{\Big(1-e^{-s_j/N}\Big)} \\
 &= e^{s_1+\cdots+s_{2k}} \prod_{i=2}^k\prod_{j=k+1}^{2k} \Big( \sum_{a_{ij} \geq 0} (e^{(-s_i-s_j)/N})^{a_{ij}} \Big) \prod_{j = k+1}^{2k}\Big( \sum_{b_j \geq 0} (e^{-s_j/N})^{b_j} \Big) \\
 &= e^{s_1+\cdots+s_{2k}} \sum_{\alpha \in \N_0^{k-1},\; \beta \in \N_0^{k}} M_{\alpha, \beta} e^{-\alpha \cdot u / N} e^{-\beta \cdot v / N},
\end{align*}
where $\cdot$ denotes the usual dot product and the $M_{\alpha, \beta} \geq 0$ are defined such that the last equation holds. The sum is convergent on the compact region of integration. Applying Fubini's theorem, we may thus interchange sum and integral to obtain
\begin{align*}
I_k &= \frac{N^{1-2k}}{(2\pi i)^{2k-1}} \sum_{\substack{\alpha \in \N_0^{k-1},\\ \beta \in \N_0^{k}}} M_{\alpha, \beta} \int_{-c_{2k} - i \pi N }^{ -c_{2k} + i \pi N}\cdots \int_{-c_{2} - i\pi N}^{-c_2 + i \pi N} e^{-\alpha \cdot u / N} e^{-\beta \cdot v / N} e^{s_2+\cdots+s_{2k}} ds_2 \cdots ds_{2k} \\
&= \frac{N^{1-2k}}{(2\pi i)^{2k-1}} \sum_{\substack{\alpha \in \N_0^{k-1},\\ \beta \in \N_0^{k}}} M_{\alpha, \beta} \Big(\int_{-c_{2k} - i\pi N}^{-c_{2k} + i \pi N} e^{-(\beta_{2k}-N) s_{2k} / N} ds_{2k}\Big) 
\cdots
\Big(\int_{-c_{2} - i\pi N}^{-c_2 + i \pi N} e^{-(\alpha_2-N) s_2 / N } ds_2\Big) \\
&= \frac{N^{1-2k}}{(2\pi i)^{2k-1}} \sum_{\substack{\alpha \in \N_0^{k-1},\\ \beta \in \N_0^{k}}} M_{\alpha, \beta} \prod_{\substack{i=2 \\ \alpha_i \neq N}}^k \frac{N\left( e^{-(\alpha_i-N)(-c_i + i\pi N)/N} - e^{-(\alpha_i - N)(-c_i - i \pi N)/N} \right)}{N - \alpha_i} \prod_{\substack{i=2 \\ \alpha_i = N}}^k (2\pi i N) \\
&\hspace{3.7cm} \times \prod_{\substack{j = k+1 \\ \beta_j \neq N}}^{2k} \frac{N\left( e^{-(\beta_j - N)(-c_j + i \pi N)/N} - e^{-(\beta_j - N)(-c_j - i \pi N) / N} \right)}{N - \beta_j} \prod_{\substack{j = k+1 \\ \beta_j = N}}^{2k} ( 2 \pi i N).
\end{align*}
Here, we let $\alpha = (\alpha_2,\dots,\alpha_k)$ and $\beta = (\beta_{k+1},\dots,\beta_{2k})$. Now, observe that $$e^{-(\alpha_i-N)(-c_i + i\pi N)/N} - e^{-(\alpha_i - N)(-c_i - i \pi N)/N} = e^{(\alpha_i/N - 1)c_i} 2i\sin(\left(-\alpha_i+N)\pi\right)),$$
and as $\alpha_i\in \Z$, this (and the corresponding expression for $\beta_j$) evaluates to 0 for all $i,j$. Thus, the only term in the summation that contributes to $I_k$ is $M_{(N, \dots, N), (N,\dots, N)}$, which is just the number of possible choices for $a_{ij}$ and $b_j$ such that
\begin{equation*}
    \sum_{j = k+1}^{2k} a_{ij} = N \quad \text{ for }\quad i=2,\dots,k, \qquad \text{and} \qquad 
    b_j + \sum_{i = 2}^k a_{ij} = N \quad \text{ for }\quad j=k+1,\dots,2k.
\end{equation*}
It follows that we also have
\begin{equation*}
\sum_{j = k+1}^{2k} b_j = \sum_{j = k+1}^{2k} \Big (b_j + \sum_{i = 2}^k a_{ij} \Big) - \sum_{i = 2}^k \sum_{j = k+1}^{2k} a_{ij} = kN - (k-1)N = N.
\end{equation*}
This gives the desired equality $I_k = \mathcal{S}_k(N)$, establishing \cref{prop:magic_square_integral}. 


\section{Proof of \cref{prop:error_term_estimation}: estimating the error term}
\label{sec:analytic-error-integrals}

To facilitate notation, we define
$$
F(w) := F(w, N) := \frac{1}{N(1-e^{-w/N})}
$$
for $w \in \C$. For $i = 2,\dots,2k$, denote 
$$
z_i = c_i + i \theta_i, \quad c_i \in \left( \frac{1}{4}, \frac{3}{4}\right), \ \theta_i \in \R.
$$
For any Lebesgue-measurable region $R \subset [-\pi N, \pi N]^{2k-1}$, we further define
$$
I(R) := \underset{R}{\int \cdots \int} \prod_{i = 2}^{k} \prod_{j = k+1}^{2k} |F(z_i + z_j)| \prod_{j = k+1}^{2k} |F(z_j)| \ d\theta_2 \cdots d\theta_{2k}
$$
and
$$\Delta(R) := \sup_{(\theta_2, ..., \theta_{2k})\in R}\left|B\left(q^{-1/2}, q^{-1/2}e^{(c_2+i\theta_2)/N}, \dots, q^{-1/2}e^{(c_{2k}+i\theta_{2k})/N}\right) - b_k(q)\right|.$$

To give a brief outline, the proof of \cref{prop:error_term_estimation} will involve splitting the region of integration, $[-\pi N, \pi N]^{2k-1}$, into three regions: 
\[
R_1 := [-\pi N, \pi N]^{2k-1}\setminus \mathfrak{S}(X), 
\qquad 
R_2 := \mathfrak{S}(X)\setminus \mathfrak{S}(X, Y),
\qquad 
R_3 := \mathfrak{S}(X, Y),
\]
as specified later by \eqref{eqn:Region-SX} and \eqref{eqn:Region-SXY}  and depend on the parameters $X$ and $Y$ given in the proposition. We observe that the integral $\widetilde{J_k}$ from \eqref{error_integral_I_prime} can be bounded by
$$I(R_1)\Delta(R_1) + I(R_2)\Delta(R_2) + I(R_3)\Delta(R_3).$$
In \cref{lem:I(R_1),lem:I(R_2)}, we will give estimates for $I(R_1)$ and $I(R_2)$, respectively. Here we take the trivial estimates for $\Delta(R_1)$ and $ \Delta(R_2)$ coming from \cref{prop:convergence_of_B}. In $R_3$, we loosely estimate $I(R_3)$ in \cref{lem:bounding_integral_trivial_estimate} and carefully bound $\Delta(R_3)$ in \cref{lem:bound_on_diff_of_b_terms}. Combining these estimates yields the desired error term in \cref{prop:error_term_estimation}. 

We start with some simple upper bounds for $F$.

\begin{lemma}{\label{lem:F_bound}}
    For $w \in \C$ with $0 < \Re(w) \leq N$, we have
    $$
    |F(w)| \ll \begin{cases}
    \frac{\textstyle 1}{\textstyle |w|} & \text{if $\ |\Im(w)| \leq \pi N$,} \\
    \frac{\textstyle 1}{\textstyle |w-2\pi i N|} & \text{if $\ \pi N \leq \Im(w) \leq 2 \pi N$,} \\
    \frac{\textstyle 1}{\textstyle |w+2\pi i N|} & \text{if $\ -2 \pi N \leq \Im(w) \leq  -\pi N$}.
    \end{cases}
    $$
    
\end{lemma}
\begin{proof}
Note that the function
$$
\frac{w}{N(1-e^{-w/N})}
$$
is holomorphic on an open set containing $R = \{ w \in \C \, | \, |\Im(w)| \leq \pi N, \, 0 \leq \Re(w) \leq N\}$ since the singularity at $0$ is removable. Making the substitution $w \mapsto N \cdot w$, we see that it suffices to show that $\frac{w}{(1-e^{-w})}$ is bounded on $\tilde{R} = \{ w \in \C\, | \, |\Im(w)| \leq \pi, \, 0 \leq \Re(w) \leq 1 \}$. This follows immediately since the above expression is holomorphic on an open set containing $\tilde{R}$. One uses the same argument in the other cases. 

\end{proof}

We will also need the following estimate:

\begin{lemma} \label{lem:bounding_integral_trivial_estimate}
    $$
    I([-\pi N, \pi N]^{2k-1}) \ll_k (\log N)^k.
    $$
\end{lemma}

\begin{proof}
For fixed $i \in \{2,\dots,k\}$, we use Hölder's inequality and the previous lemma to obtain
\begin{align*}
\int_{-\pi N}^{\pi N} \prod_{j = k+1}^{2k} |F(z_i + z_j)|d\theta_i &\leq \prod_{j = k+1}^{2k} \Big( \int_{-\pi N}^{\pi N} |F(z_i + z_j)|^k d\theta_i \Big)^{1/k} \\ 
&\ll_k \prod_{j = k+1}^{2k} \Big( \int_{|\theta_i + \theta_j| \leq \pi N} \frac{1}{|z_i+z_j|^k} d\theta_i \Big. \\ &\hspace{1.1cm} + \int_{-2\pi N \leq \theta_i + \theta_j \leq -\pi N} \frac{1}{|z_i + z_j + 2\pi i N |^k} d \theta_i \\ &\hspace{1.1cm} + \Big. \int_{ \pi N \leq \theta_i + \theta_j \leq 2 \pi N} \frac{1}{|z_i + z_j - 2\pi i N|^k} d \theta_i \Big)^{1/k}. \\
\end{align*}
Upon applying the inequality $1/|w| \leq \sqrt{2}/(\Re(w)+|\Im(w)|)$ with $\Re(w) > 0$, we have that the above is
\begin{align*}
&\ll \prod_{j = k+1}^{2k} \Big( \int_{\R} \frac{1}{(c_i+c_j+|\theta_i+\theta_j|)^k} d\theta_i \Big. 
+ \int_{\R} \frac{1}{(c_i+c_j+|\theta_i + \theta_j + 2\pi i N |)^k} d \theta_i 
\\ &\hspace{1.5cm} + \Big. \int_{\R} \frac{1}{(c_i+c_j + |\theta_i + \theta_j - 2\pi i N|)^k} d \theta_i \Big)^{1/k}, \\
\end{align*}
and this is clearly bounded by a constant independent of $N$ for $k \geq 2$. (Recall that we made this assumption at the very beginning of \cref{sec:analytic-general}).
Thus, 
\begin{align*}
    I([-\pi N, \pi N]^{2k-1}) &\ll_k \int_{-\pi N}^{\pi N} \cdots \int_{- \pi N}^{\pi N} \prod_{j = k+1}^{2k} |F(z_j)| d \theta_{k+1} \cdots d \theta_{2k} \\
    &= \prod_{j = k+1}^{2k} \int_{-\pi N}^{\pi N} |F(z_j)| d \theta_j \\
    &\ll \prod_{j = k+1}^{2k} \int_{0}^{\pi N} \frac{1}{c_j+\theta_j} d \theta_j \\
    &\ll \prod_{j = k+1}^{2k} \log N = (\log N)^k. \qedhere
\end{align*} 
\end{proof}
The next few propositions estimate different regions of the integral $I$. Let $2 \leq X \leq N/10$ be a parameter. Define the set $\mathfrak{S}(X) \subseteq [-\pi N, \pi N]^{2k-1}$ by
\begin{equation}
\{ (\theta_2,\dots,\theta_{2k}) \in [-\pi N, \pi N]^{2k-1}\ \big|\ |z_i + z_j| \leq X \text{ for } 2 \leq i \leq k, \ k+1 \leq j \leq 2k\}.
\label{eqn:Region-SX}
\end{equation}
We begin by estimating the contribution outside $\mathfrak{S}(X)$. 
\begin{lemma}{\label{lem:I(R_1)}}
    $$
    I([-\pi N, \pi N]^{2k-1} \setminus \mathfrak{S}(X)) \ll_k \frac{(\log N)^{2k-1}}{X^{1-1/k}} + (\log N)^{2k-2} \frac{X}{N}.
    $$
\end{lemma}

\begin{proof}
    Let $\Omega := \{2,\dots,k\} \times \{k+1,\dots,2k\}$ with $|\Omega| = k(k-1)$. We say a $k$-partition $\Lambda_1 \sqcup \Lambda_2 \sqcup \cdots \sqcup \Lambda_k$ of $\Omega$ is \textit{weak} if we allow $\Lambda_i = \emptyset$. For every weak 3-partition $\Lambda = (\Lambda_1,\Lambda_2,\Lambda_3)$ of $\Omega$, define the region $R_{\Lambda}(X)$ to be the set 
    {\small 
    \begin{align*}
    \left\{(\theta_2, \dots, \theta_{2k}) \in [-\pi N, \pi N]^{2k-1} \Bigg| \begin{array}{cc}
        |z_i + z_j| \leq X & \forall\ (i,j) \in \Lambda_1, \\
        |z_i + z_j - 2\pi i N| \leq X \textnormal{ or } |z_i + z_j + 2\pi i N| \leq X & \forall\ (i,j) \in \Lambda_2, \\
        |z_i + z_j| > X \textnormal{ and } |z_i + z_j \pm 2\pi i N| > X & \forall\ (i,j) \in \Lambda_3
    \end{array}\right\}.
    \end{align*}}
    It is easy to see that $[-\pi N, \pi N]^{2k-1}$ is partitioned by $\{R_\Lambda\ |\ \Lambda \textnormal{ a weak 3-partition of }\Omega\}$ and that $\mathfrak{S}(X) = R_{(\Omega, \emptyset, \emptyset)}(X)$. We further define two projection maps $\pi_i, \ i=1,2$ which respectively return the first and second components of a pair of integers, that is, $\pi_1(a,b) = a$ and $\pi_2(a,b) = b$ for $a,b \in \Z$.  
    
    It now suffices to prove that
    \begin{equation} \label{eq:I(R)}
    I(R_\Lambda(X)) \ll_k \frac{(\log N)^{2k-1-|\pi_2(\Lambda_2)|}}{X^{|\Lambda_3|(1-1/k)}} \Big(\frac{X}{N}\Big)^{|\pi_2(\Lambda_2)|},
    \end{equation}
    because if $\Lambda_2 \neq \emptyset$, then the RHS of the above expression is $\ll_k (\log N)^{2k-2} X/N$. If, on the other hand, $\Lambda_3 \neq \emptyset$, then the RHS is $\ll_k {(\log N)^{2k-1}}/{X^{1-1/k}}$, and, since $\mathfrak{S}(X)=R_{(\Omega, \emptyset, \emptyset)}(X)$, at least one of these two cases must occur. 
    Now, fix a weak partition $\Lambda$ of $\Omega$ and write $R_\Lambda := R_\Lambda(X)$. For $(\theta_2,\dots,\theta_{2k}) \in R_\Lambda$, notice that
    $$
    \prod_{i = 2}^{k} \prod_{j = k+1}^{2k} |F(z_i + z_j)| \ll_k \frac{1}{X^{|\Lambda_3|}}
    $$
    by assumption and \cref{lem:F_bound}. Hence, 
    $$
    I(R_\Lambda) \ll_k \frac{1}{X^{|\Lambda_3|(1-1/k)}} \underset{R_\Lambda}{\int \cdots \int} \prod_{i = 2}^{k} \prod_{j = k+1}^{2k} |F(z_i + z_j)|^{\frac{1}{k}}\ \prod_{j = k+1}^{2k} |F(z_j)|\ d\theta_2 \cdots d\theta_{2k}.
    $$
    Further note that $R_\Lambda \subset [-\pi N, \pi N]^{k-1} \times \tilde{R}_\Lambda$ where
    $$
    \tilde{R}_\Lambda := \{ (\theta_{k+1},\dots,\theta_{2k}) \in [-\pi N, \pi N]^{k}\ \big|\  \pi N - X \leq |\theta_j| \leq \pi N\; \forall\ j \in \pi_2(\Lambda_2)  \}.
    $$
    This inclusion holds because if we had $|\theta_j| < \pi N - X$ for $(i,j) \in \Lambda_2$, then
    $$
    |z_i + z_j \pm 2 \pi i N| \geq \left| |\theta_i + \theta_j| - 2 \pi N \right| \geq 2\pi N - |\theta_i| - |\theta_j| > 2 \pi N - \pi N + X - \pi N = X.
    $$
    Consequently, the above gives
    $$
    I(R_\Lambda) \ll_k \frac{1}{X^{|\Lambda_3|(1-1/k)}} \underset{[-\pi N, \pi N]^{k-1} \times \tilde{R}_\Lambda}{\int \cdots \int} \prod_{i = 2}^{k} \prod_{j = k+1}^{2k} |F(z_i + z_j)|^{\frac{1}{k}}\ \prod_{j = k+1}^{2k} |F(z_j)|\ d\theta_2 \cdots d\theta_{2k}.
    $$
    By Hölder's inequality,
    $$
    \prod_{i = 2}^{k} \int_{- \pi N}^{\pi N} \prod_{j = k+1}^{2k} |F(z_i + z_j)|^{\frac{1}{k}} d\theta_i \leq \prod_{i = 2}^{k} \prod_{j = k+1}^{2k} \Big(\underbrace{\int_{- \pi N}^{\pi N} |F(z_i + z_j)| d\theta_i}_{\ll \log N} \Big)^{\frac{1}{k}} \ll (\log N)^{k-1}.
    $$
    Finally, we obtain
    \begin{align*}
    I(R_\Lambda) &\ll_k \frac{(\log N)^{k-1}}{X^{|\Lambda_3|(1-1/k)}} \underset{\tilde{R}_\Lambda}{\int \cdots \int} \prod_{j = k+1}^{2k} |F(z_j)| d \theta_{k+1} \cdots d \theta_{2k} \\
    &\ll_k \frac{(\log N)^{k-1}}{X^{|\Lambda_3|(1-1/k)}} \prod_{j \notin \pi_2(\Lambda_2)} \underbrace{\int_{- \pi N}^{\pi N} |F(z_i + z_j)| d\theta_i}_{\ll \log N}   \prod_{j \in \pi_2(\Lambda_2)}  \Big(\int_{-\pi N}^{-\pi N + X} + \int_{\pi N - X}^{\pi N}\Big) \underbrace{|F(z_j)|}_{\ll \frac{1}{N}} d \theta_j  \\
    &\ll \frac{(\log N)^{2k-1-|\pi_2(\Lambda_2)|}}{X^{|\Lambda_3|(1-1/k)}} \Big( \frac{X}{N}\Big)^{|\pi_2(\Lambda_2)|}.
    \end{align*}
    Thus, we have shown \eqref{eq:I(R)}. 
\end{proof}

Let $10X \leq Y \leq N$ be a parameter. Define the region $\mathfrak{S}(X,Y) \subseteq \mathfrak{S}(X)$ to be the set 
\begin{equation}
\left\{ (\theta_2,\dots,\theta_{2k}) \in [-\pi N, \pi N]^{2k-1}\ \bigg|\ \begin{array}{cc}
   |z_i + z_j| \leq X  &  \forall\ 2 \leq i \leq k, \ k+1 \leq j \leq 2k, \\
   |z_j| \leq Y  & \forall\ k+1 \leq j \leq 2k
\end{array} \right\}.
\label{eqn:Region-SXY}
\end{equation}

\begin{lemma}{\label{lem:I(R_2)}}
    We have
    $$
    I(\mathfrak{S}(X) \setminus \mathfrak{S}(X,Y)) \ll_k \Big( \frac{X}{Y} \Big)^{k-1}.
    $$
\end{lemma}

\begin{proof}
For $l \in \{ k+1,\dots,2k\}$, we define
$$
\mathfrak{T}_l(X,Y) := \left\{(\theta_2,\dots,\theta_{2k}) \in [-\pi N, \pi N]^{2k-1}\ \bigg|\ 
\begin{array}{cc}
    |z_i + z_j| \leq X & \forall\ 2 \leq i \leq k, \ k+1 \leq j \leq 2k, \\
     |z_l | > Y & 
\end{array}\right\}.
$$
We then have $\mathfrak{S}(X) \setminus \mathfrak{S}(X,Y) \subset \bigcup_{l = k+1}^{2k} \mathfrak{T}_l(X,Y)$, and it suffices to prove
$$
I(\mathfrak{T}_l(X,Y)) \ll_k \Big( \frac{X}{Y} \Big)^{k-1} \quad \forall\ l = k+1,\dots,2k.
$$
We claim that 
$$
\mathfrak{T}_l(X,Y) \subset [-\pi N, \pi N]^{k-1} \times \tilde{\mathfrak{T}_l}(X,Y),
$$
where
$$
\tilde{\mathfrak{T}}_l(X,Y) := \left\{(\theta_2,\dots,\theta_{2k}) \in [-\pi N, \pi N]^{2k-1}\ \bigg|\ 
\begin{array}{cc}
    |\theta_j - \theta_l| \leq 2X & \forall\ k+1 \leq j \leq 2k, \\
     |\theta_l| \geq Y/2 & 
\end{array}\right\}.
$$
To prove this, let $(\theta_2,\dots,\theta_{2k}) \in \mathfrak{T}_l(X,Y)$. As $|z_l| > Y$, it follows that 
$$
|\theta_l| \geq Y - |c_l| \geq Y - \frac{3}{4} \geq Y - \frac{X}{2} \geq \frac{Y}{2}.$$ By the triangle inequality, we also have
$$
|\theta_j - \theta_l| \leq |z_j - z_l| \leq |z_j + z_2| + |z_2 + z_l| \leq 2X. 
$$
This proves the claim. Therefore,
{\small
\begin{align*}
    I(\mathfrak{T}_l(X,Y)) &\leq I([-\pi N, \pi N]^{k-1}\times \tilde{\mathfrak{T}}_l(X,Y)) \\
    &= \underset{\tilde{\mathfrak{T}}_l(X,Y)}{\int \cdots \int} \Big( \underbrace{\int_{-\pi N}^{\pi N} \cdots \int_{-\pi N}^{\pi N} \prod_{i = 2}^k \prod_{j = k+1}^{2k} |F(z_i + z_j)| d \theta_i}_{\ll_k 1 \textnormal{ by Hölder as in (\cref{lem:bounding_integral_trivial_estimate})}} \Big)  \prod_{j = k+1}^{2k} |F(z_j)|\ d \theta_{k+1} \cdots d \theta_{2k} \\
    &\ll_k \underset{\tilde{\mathfrak{T}}_l(X,Y)}{\int \cdots \int} \prod_{j = k+1}^{2k} |F(z_j)|\ d \theta_{k+1} \cdots d \theta_{2k} \\
    &\ll_k \int_{\frac{Y}{2} \leq |\theta_l| \leq \pi N} \bigg( \prod_{\substack{j = k+1 \\ j \neq l}}^{2k} \int_{|\theta_j - \theta_l| \leq 2X} |F(z_j)|\ d \theta_j \bigg) |F(z_l)|\ d \theta_l.
\end{align*}}
Note that for $j \neq l$, $|F(z_j)| \ll |\theta_j|^{-1} \ll |\theta_l|^{-1}$ where the first bound follows from Lemma \ref{lem:F_bound} and the second from
$$
|\theta_j| \geq |\theta_l| - |\theta_j - \theta_l|
$$
and
$$
|\theta_j - \theta_l| \leq 2X \leq \frac{Y}{5} \leq \frac{|\theta_l|}{2}.
$$
Furthermore, the length of the interval of integration for each $j$ is $O(X)$. Thus, 
\begin{align*}
     I(\mathfrak{T}_l(X,Y)) &\ll_k \int_{|\theta_l| \geq \frac{Y}{2}} \bigg( \underbrace{|F(z_l)|}_{\ll \frac{1}{|\theta_l|}} \prod_{\substack{j = k+1 \\ j \neq l}}^{2k} \frac{X}{|\theta_l|} d \theta_l \bigg) \\
     &\ll_k \int_{|\theta_l| \geq \frac{Y}{2}} \frac{X^{k-1}}{|\theta_l|^{k}} d\theta_l \ll_k \Big( \frac{X}{Y} \Big)^{k-1}.
\end{align*}
This proves the lemma.

\end{proof}

For the remaining region $\mathfrak{S}(X,Y)$, we estimate trivially
$$
I(\mathfrak{S}(X,Y)) \ll_k (\log N)^k.
$$
Here, our savings will come from the fact that the difference
$$B\left(q^{-1/2}, q^{-1/2}e^{(c_2+i\theta_2)/N}, \dots, q^{-1/2}e^{(c_{2k}+i\theta_{2k})/N}\right) - b_k(q) $$
is small on this region, where all of the exponents $z_i+z_j$ and $z_j$ are small. The next lemma makes this precise.

\begin{lemma} \label{lem:bound_on_diff_of_b_terms} We have
$$\Delta([-\pi N, \pi N]^{2k-1}) \ll_{k} \frac{1}{q}.$$
and
$$\Delta(\mathfrak{S}(X,Y)) \ll_{k} \frac{Y}{qN}.$$
\end{lemma}

\begin{proof} 
Note that the first bound is an immediate consequence of \cref{prop:convergence_of_B} and
$$
\max_{-\pi N \leq \theta_i \leq N \pi}|e^{c_i/N + i \theta_i/N}-1| \ll 1.  
$$

If we have $(\theta_2,\dots,\theta_{2k}) \in \mathfrak{S}(X,Y)$, then $|\theta_i| \leq 2Y$ for all $i=2,\dots,2k$. Thus, by \cref{prop:convergence_of_B}, it again suffices to show that
$$
\max_{(\theta_2,\dots,\theta_{2k}) \in \mathfrak{S}(X,Y)} ||(q^{-1/2}e^{c_2/N + i\theta_2/N}, \dots, q^{-1/2}e^{c_{2k}/N + i\theta_{2k}/N})-(q^{-1/2},\dots,q^{-1/2})||_\infty \ll \frac{Y}{\sqrt{q} N}.
$$

Using $|\theta_i| \leq 2Y$, we can simplify this statement to
\begin{equation}\label{equation_2Y/N}
\max_{-2Y \leq \theta_i \leq 2Y}|e^{c_i/N + i \theta_i/N}-1| \ll \frac{Y}{N}  
\end{equation}
for each $i$.
First, note that
$$
\lim_{x \rightarrow 0} \frac{e^{c_i x}-1}{c_i x} = 1.
$$
This implies
$$
|e^{c_i/N}-1| \ll_{c_i} \frac{1}{N}.
$$
Furthermore, for any $\theta \in \R$, we have
$$
|e^{i\theta}-1| \leq |\theta|.
$$
Putting those two inequalities together, we conclude
\begin{align*}
|e^{c_i/N + i\theta_i/N}-1| &= |e^{c_i/N + i\theta_i/N}-e^{i\theta_i/N}+e^{i\theta_i/N}-1| \leq |e^{c_i/N}-1| + |e^{i\theta_i/N}-1| \\ &\ll_{c_i} \frac{1}{N} ( 1 + |\theta|).
\end{align*}
We now take the maximum over $-2Y \leq \theta \leq 2Y$ to establish \eqref{equation_2Y/N}.
\end{proof}

Thus, we obtain the desired bound on the integral $\widetilde{J_k}$:

\begin{align*}
|\widetilde{J_k}| &\ll_k \frac{1}{q} \frac{(\log N)^{2k-1}}{X^{1-1/k}} + \frac{1}{q} (\log N)^{2k-2} \frac{X}{N} + \frac{1}{q} \left( \frac{X}{Y} \right)^{k-1} +   \frac{Y}{q N} \left( \log N \right)^k.
\end{align*}
The first two terms are a consequence of estimating $I(R_1)$ using \cref{lem:I(R_1)} and taking the trivial $1/q$ estimate for $\Delta(R_1)$. The third term makes use of \cref{lem:I(R_2)} to bound $I(R_2)$ and uses the same estimate for $\Delta(R_2)$. Lastly, the fourth term comes from upper bounding $I(R_3)$ by $I([-\pi N, \pi N]^{2k-1})$ and using \cref{lem:bound_on_diff_of_b_terms} to bound $\Delta(R_3)$. This completes the proof of \cref{prop:error_term_estimation} and thus the proof of \cref{thm:steinhaus}. 

\section*{Acknowledgements}

The authors were participants in the Fields Undergraduate Summer Research Program 2023 and are deeply grateful to the Fields Institute for its support. We are thankful to Ofir Gorodetsky for many insightful mathematical and historical remarks on an earlier version of the paper, which improved the exposition and noted our proof gave better $q$-dependence in \cref{thm:steinhaus} than we had initially stated.  We also thank Andrew Granville and Winston Heap for helpful comments and their encouragement.

\bibliographystyle{alpha}
\bibliography{references}

\appendix

\end{document}

%% file: paper-style.tex
\usepackage{fullpage}

\usepackage{amssymb}
\usepackage{amsfonts}
\usepackage{amsmath}
\usepackage{mathrsfs} 
\usepackage{mathtools}

\usepackage{enumerate}
\usepackage[shortlabels]{enumitem}

\usepackage{hyperref} 
\usepackage[noabbrev,capitalise]{cleveref} 
\usepackage{url}

\usepackage{bbm}
\usepackage{comment}
\usepackage{thm-restate}
\usepackage{xcolor}

\newtheorem*{corollary*}{Corollary}
\newtheorem*{theorem*}{Theorem}

\newtheorem{theorem}{Theorem}[section]

\newtheorem{proposition}[theorem]{Proposition}
\newtheorem{lemma}[theorem]{Lemma}

\theoremstyle{definition}

\theoremstyle{remark}

\newtheorem*{remark}{Remark}

\numberwithin{equation}{section}

\newcommand{\declarecommand}[1]{\providecommand{#1}{}\renewcommand{#1}}

\declarecommand{\R}{\mathbb{R}}
\declarecommand{\Q}{\mathbb{Q}}
\declarecommand{\Z}{\mathbb{Z}}
\declarecommand{\N}{\mathbb{N}}
\declarecommand{\C}{\mathbb{C}}
\declarecommand{\emptyset}{\varnothing}

\declarecommand{\Re}{\mathrm{Re}}
\declarecommand{\Im}{\mathrm{Im}}

\declarecommand{\epsilon}{\varepsilon}
\declarecommand{\cM}{\mathcal{M}}
\declarecommand{\ds}{\displaystyle}
\newcommand{\dsum}{\mathop{\sum\sum}}
\newcommand{\qsum}{\mathop{\sum\sum\sum\sum}}
\newcommand{\msum}{\mathop{\sum\cdots\sum}}